\documentclass[a4paper,reqno]{amsart}

\usepackage{amsthm}
\usepackage{amssymb}
\usepackage{amsmath}
\usepackage[dvipsnames]{xcolor}
    \definecolor{awesome}{rgb}{1.0, 0.13, 0.32}
    \definecolor{Mulberry}{rgb}{0.66, 0.24, 0.58}
\usepackage{graphicx}
\usepackage{tikz}
\usepackage{fullpage}
\usepackage[english]{babel}
\usepackage{bm}
\usepackage{txfonts}
\usepackage{comment}
\usepackage{enumitem}
\usepackage{dsfont}
    
\usepackage{microtype}
\usepackage[colorlinks = true, linkbordercolor = {white}, urlcolor={purple}, linkcolor={purple}, citecolor = {purple}]{hyperref}
\usepackage{cleveref}
\usepackage[foot]{amsaddr}
\usepackage{marginnote}
\usepackage{mathrsfs}
\usepackage{subcaption}

\numberwithin{equation}{section}

\author[S.\,van~Golden]{S.\,van~Golden}
\author[C.\,Kalle]{C.\,Kalle}
\author[S.\,Kombrink]{S.\,Kombrink}
\author[T.\,Samuel]{T.\,Samuel}

\address[S.\,van~Golden]{Institute of Mathematics, Universiteit Leiden, Netherlands and School of Mathematics, University of Birmingham, UK}
\address[C.\,Kalle]{Institute of Mathematics, Universiteit Leiden, Netherlands.}
\address[S.\,Kombrink]{School of Mathematics, University of Birmingham, UK}
\address[T.\,Samuel]{
\textls[-15]{Department of Mathematics and Statistics, University of Exeter, UK and School of Mathematics, University of Birmingham, UK}}

\thanks{\emph{Funding acknowledgements}. The \emph{Birmingham-Leiden Collaboration Fund}, and EPSRC grants EP/S02297X/1 and EP/Y023358/1.}

\newcommand{\diam}{\operatorname{diam}}

\renewcommand{\emph}{\textsl}
\renewcommand{\em}{\textsl}
\renewcommand{\textit}{\textsl}

\numberwithin{equation}{section}

\subjclass[2020]{28A80, 11A67, 11K55}

\keywords{Affine iterated function system, Lüroth expansion, restricted digit set}

\title{Dimensions of infinitely generated self-affine sets and restricted digit sets for signed Lüroth expansions}

    \newtheorem{theorem}{Theorem}[section]
    \newtheorem{prop}[theorem]{Proposition}
    \newtheorem{lemma}[theorem]{Lemma}
    \newtheorem{cor}[theorem]{Corollary}
\theoremstyle{definition}
    
\theoremstyle{remark}
    
    \newtheorem{remark}[theorem]{Remark}
    \newtheorem{ex}[theorem]{Example}

\Crefname{theorem}{Theorem}{Theorems}
\Crefname{prop}{Proposition}{Propositions}
\Crefname{lemma}{Lemma}{Lemma}
\Crefname{cor}{Corollary}{Corollaries}
\Crefname{defn}{Definition}{Definitions}
\Crefname{question}{Question}{Questions}
\Crefname{remark}{Remark}{Remarks}
\Crefname{ex}{Example}{Examples}

\begin{document}

\begin{abstract}
    For countably infinite IFSs on $\mathbb R^2$ consisting of affine contractions with diagonal linear parts, we give conditions under which the affinity dimension is an upper bound for the Hausdorff dimension and a lower bound for the lower box-counting dimension. Moreover, we identify a family of countably infinite IFSs for which the Hausdorff and affinity dimension are equal, and which have full dimension spectrum. The corresponding self-affine sets are related to restricted digit sets for signed L\"uroth expansions.
\end{abstract}

\maketitle

\section{\textbf{Introduction and statement of main results}}

The dimension theory of self-affine sets generated by finite iterated function systems (IFS) has been developed since the 1980s, when it was investigated for which types of sets the Hausdorff and box-counting dimensions coincide, see for example \cite{PhD_Bedford,MR0771063}. In 1988 Falconer \cite{falconer_1988} introduced the affinity dimension, a dimension formula which purely depends on the singular values of the linear parts of the affine maps in the IFS. It turns out that for finitely generated self-affine sets in $\mathbb{R}^{D}$, the affinity dimension is an upper bound for the upper box-counting dimension, which is known to be an upper bound for the Hausdorff dimension. Moreover, Falconer proved that the Hausdorff dimension is almost surely (with respect to the translation vectors of the affine maps in the IFS) equal to the minimum of $D$ and the affinity dimension. Fraser \cite{Fraser_2012} later introduced a modified affinity dimension and showed, for a class of finitely generated box-like self-affine sets satisfying the rectangular open set condition, that the modified affinity dimension coincides with the box-counting and packing dimensions. More recently, Morris \cite{Morris} gave a simple description of the affinity dimension of self-affine sets in case the linear parts of the contractions consist of diagonal and anti-diagonal matrices. In the diagonal case, under the condition that each of the canonical projections of the IFS is exponentially separated, Rapaport \cite{Rap_2023} showed that the Hausdorff dimension of the self-affine set coincides with the minimum of its affinity dimension and $D$. The authors of \cite{KR14,Jurga_dimension_spectrum} considered the affinity dimension for infinite affine IFSs that are \emph{irreducible}, meaning the linear parts of the affine maps do not all preserve a common proper non-trivial linear subspace. They showed that the Hausdorff and affinity dimensions of the corresponding self-affine sets coincide. Outside of these classes of self-affine sets the Hausdorff and affinity dimensions do not necessarily coincide. For instance the Hausdorff dimension of the self-affine set constructed by taking the cross product of the middle $\frac{1}{2}$-Cantor set and the middle $\frac{7}{8}$-Cantor set equals $\frac{3}{4}$, whereas its affinity dimension is $1$.\!\footnote{\,We thank Ian Morris for providing this example.}

Conformal infinite IFSs have been studied since the seminal work of Mauldin and Urba\'nski \cite{mauldin1996dimensions}. One of the differences between finite and infinite conformal IFSs, highlighted in \cite{mauldin1996dimensions}, is that even under the open set condition, the Hausdorff and box-counting dimensions of their limit sets need not be equal. Moreover, in contrast to the finite setting, the limit set of an infinite IFS need not be compact. 

In this article we generalise some of the above results to a class of non-irreducible non-conformal infinite IFSs, for which the projections are not necessarily exponentially separated. More precisely, for a countable collection $\{ L_i : i \in I\}$ of diagonal $2 \times 2$~matrices over $\mathbb{R}$ of the form
    \begin{align}\label{eq:Li}
        L_i = \begin{pmatrix}
            a_i & 0\\
            0 & b_i
        \end{pmatrix},
        \quad \lvert a_i \rvert, \lvert b_i \rvert\in (0,1),
    \end{align}
and for $r > 0$, we set
\begin{align}\label{eq:pressure}
    P_I(r) = 
        \begin{cases}
            \max \{\sum_{i\in I} \lvert a_i \rvert^r,  \sum_{i\in I} \lvert b_i \rvert^r \} & \text{if} \; 0 < r \leq 1,\\[0.25em]
            \max \{\sum_{i\in I} \lvert a_i \rvert \cdot \lvert b_i \rvert^{r-1},  \sum_{i\in I} \lvert b_i \rvert \cdot \lvert a_i \rvert^{r-1} \} & \text{if} \; 1 < r \leq 2,\\[0.25em]
            \sum_{i\in I} \lvert a_i\cdot b_i \rvert^{r/2} & \text{if} \; r > 2.
        \end{cases}
\end{align}
The \emph{affinity dimension} $d( L_{i} \ \vert \ i \in I)$ of $\{ L_i : i \in I \}$ is defined by
    \begin{align*}
        d( L_{i} \ \vert \ i \in I) = \inf \left\{ r >0 \, : \, \sum_{m \in \mathbb{N}} \sum_{u \in I^{m}} \varphi^{r}(L_{u}) < \infty \right\},
    \end{align*}
where $\varphi^r(L_{u}) = \varphi^r(L_{(u_1, \ldots, u_m)})$ is the singular value function of the matrix product $L_{u_1} \cdots L_{u_m}$, see \eqref{eq:singular-value-fcn}. For countably infinite alphabets $I$ we find the following analogue of \cite[Corollary~2]{Morris}.

\begin{theorem}\label{theorem:MorrisInfinite}
    Suppose we have a countable alphabet $I$ and a collection $\{L_i : i\in I\}$ of diagonal $2\times 2$~matrices, as given in \eqref{eq:Li}, with $\sup_{i\in I} \max\{ \lvert a_i \rvert, \lvert b_i \rvert \} < 1$, then
        \begin{align*}
            \inf \left\{ r > 0 \, : \, \sum_{m \in \mathbb{N}} \sum_{u \in I^{m}} \varphi^{r}(L_{u}) < \infty \right\}
            = \inf\{r> 0 : P_I(r) \leq 1\}.
        \end{align*}
\end{theorem}

Our next result compares different notions of dimension to the affinity dimension and identifies a large class of infinite IFSs for which the affinity dimension gives a lower bound for the lower box-counting dimension and an upper bound for the Hausdorff dimension.

\begin{theorem} \label{prop:lower bound}
    Let $I$ be a countably infinite alphabet and $F$ be the self-affine set of an IFS $\{A_i : i\in I\}$ on $[0,1]^2$, where each $A_i$ is an affine map with linear part $L_{i}$ as in \eqref{eq:Li}, and $\sup_{i\in I} \max\{ \lvert a_i \rvert, \lvert b_i \rvert \} < 1$. Then the following hold.
        \begin{enumerate}
            \item\label{item:prop:lower_bound_1}  $\dim_{\mathcal{H}}(F) \leq \min\{2,d(L_i \ \vert \ i \in I )\}$;
            \item\label{item:prop:lower_bound_2} If there exists a finite alphabet $I_1\subseteq I$ such that $\dim_B(F_{I_2}) = d(L_i \ \vert \ i \in I_2)$ for all finite alphabets $I_{2}$ with $I_1\subseteq I_2 \subseteq I$, where $F_{I_2}$ is the limit set  of $\{A_i : i\in I_2\}$, then $d(L_i \ \vert \ i\in I) \leq \underline{\dim}_{B}(F) \leq \overline{\dim}_{B}(F) = \dim_{P}(F)$.
        \end{enumerate}
\end{theorem}

Let $\pi_1, \pi_2$ denote the canonical projections onto the first and second coordinate, respectively. The conditions in \Cref{prop:lower bound}\eqref{item:prop:lower_bound_2} hold, for example, when $F$ satisfies the rectangular open set condition (ROSC), see \Cref{s:sss}, and there exists a finite subalphabet $I_1\subseteq I$ such that either
    \begin{enumerate}
        \item $\dim_{B}(\pi_1(F_{I_1})) = \dim_{B}(\pi_2(F_{I_1})) = 1$, or
        \item $\dim_{B}(\pi_1(F_{I_1})) = 1$ and $\lvert a_i \rvert \geq \lvert b_i \rvert$ for each $i\in I$.
    \end{enumerate}
This is a consequence of \cite[Corollaries~2.6 and~2.7]{Fraser_2012} with the fact that if $I_{1} \subseteq I_{2} \subseteq I$, then $F_{I_{1}} \subseteq F_{I_{2}} \subseteq F_{I}$, and if $\{A_i : i\in I\}$ is an infinite IFS satisfying the ROSC, then $\{A_i : i\in I'\}$ satisfies the ROSC for $I' \subseteq I$.

Next we provide a family of planar self-affine sets for which we can simplify the affinity dimension even further. Considering this family is motivated by questions on number expansions with restrictions on their digits. A famous example of a restricted digit set is the middle third Cantor set, which is the set of numbers in $[0,1]$ that have a ternary expansion without the digit $1$. For non-integer base expansions results on restricted digit sets are considered, for example, in \cite{KSS95,PS95,Lal97,DK09}. For continued fractions, which have infinite digit sets, restricted digit sets have been extensively studied since the work of Jarnik \cite{Jar28} and Good \cite{Goo41}. The infinite IFSs we will be concerned with relate to another type of number expansions with an infinite digit set, namely \emph{Lüroth expansions} \cite{Luroth1883}. For $x \in (0,1]$, these are expressions of the form 
    \begin{align}\label{q:lurothexp}
        x = \sum_{n \in \mathbb{N}} \frac{1}{d_n(d_n-1)},
    \end{align}
where $d_n \in \mathbb{N}_{\geq 2}$ for $n \in \mathbb{N}$. There are many known results on Lüroth expansions, for instance, concerning level sets defined in terms of the frequencies of digits or sets of points with growth rate restrictions on the digits, see for example \cite{BI09,FLMW10,LWY18,AG21,Zho22,FZ23, BK23}. 

Lüroth expansions can be obtained from the infinite IFS $\{ h_{d} : d \in \mathbb{N}_{\geq 2} \}$ where $h_d:[0,1] \to [1/d, 1/(d-1)]$ is defined by $h_d(x) = (x+d)/(d(d-1))$. If $x$ has a Lüroth expansion as in \eqref{q:lurothexp} with digit sequence $(d_n)_{n \in \mathbb{N}}$, then
    \begin{align*} 
        x = \lim_{n \to \infty} h_{d_1} \circ h_{d_2} \circ \cdots \circ h_{d_n}(0).
    \end{align*}
Over the years several generalisations of the Lüroth number system have been proposed. In particular, the authors of  \cite{KKK90,KKK91} considered alternating Lüroth expansions, which are very similar to the ones in \eqref{q:lurothexp} but the terms in the series alternate in sign, hence the name. In \cite{BBDK1996} it was shown that alternating Lüroth expansions have better approximation properties than Lüroth expansions and a family of number systems was described that interpolate between the Lüroth and alternating Lüroth systems. The corresponding expansions, which we call \emph{signed Lüroth expansions}, are of the form 
    \begin{align}\label{q:genluroth}
        \sum_{n \in \mathbb{N}} (-1)^{\sum_{i=1}^{n-1}s_i} \frac{d_n-1+s_n}{\prod_{i=1}^n d_i(d_i-1)}, 
    \end{align}
where $s_n \in \{0,1\}$ and $d_n \in \mathbb{N}_{\geq 2}$ for $n \in \mathbb{N}$, and where we set $\sum_{i=1}^{0}s_i = 0$. In \cite{KM22b} it was shown that Lebesgue almost all numbers $x \in [0,1]$ have uncountably many different signed Lüroth expansions and a one-parameter family of number systems in $\mathbb{R}^{2}$ was introduced that generate, for each $x$, all possible signed Lüroth expansions.

The system from \cite{KM22b} is related to an infinite affine IFS as follows. For each parameter $p \in (0,1)$ consider the IFS $\{ A_{s,d}^p: (s,d) \in \{0,1\} \times \mathbb{N}_{\geq 2}  \}$ where $A_{s,d}^p : [0,1]^{2} \to [0,1]^{2}$ is defined, for $(w, x) \in [0,1]^{2}$, by 
    \begin{align}\label{q:Land}
        A_{s,d}^p (w, x) = (L_{s,d}^p (w,x)^{\top}+ v_{s,d}^p)^{\top}
        \quad \text{with} \quad
        L_{s,d}^p = \begin{pmatrix} 
            p^{1-s}(1-p)^s & 0\\
            0 & (-1)^s \tfrac{1}{d(d-1)}
        \end{pmatrix}
        \quad \text{and} \quad
        v_{s,d}^p = \begin{pmatrix} 
            sp\\
            \tfrac{1}{d-s}
            \end{pmatrix}.
    \end{align}
For $p\in (0,1)$ and $J\subseteq \{0,1\}\times \mathbb N_{\geq 2}$ we let $\mathcal F_J^p$ denote the self-affine set of the non-irreducible IFS $\{ A_{s,d}^p: (s,d) \in J  \}$. Approximate images of examples of the sets $\mathcal F_{\!J}^p$ for $p=\tfrac12$ are shown in Figure \ref{fig:the sets}.

If $x$ has an expansion of the form \eqref{q:genluroth} with sign sequence $(s_n)_{n \in \mathbb{N}} \in \{0,1\}^\mathbb{N}$ and digit sequence $(d_n)_{n \in \mathbb{N}} \in \mathbb{N}^{\mathbb{N}}_{\geq 2}$, then $x = \pi_2 (\lim_{n \to \infty} A_{s_1,d_1}^p \circ A_{s_2, d_2}^p \circ \cdots \circ A_{s_n,d_n}^p((0,0)))$. The maps $A_{0,d}^p$ correspond to the Lüroth system in the sense that for each digit sequence $(d_n)_{n \in \mathbb{N}}$,
    \begin{align*}
        \pi_2 \left(\lim_{n \to \infty} A_{0,d_1}^p \circ A_{0, d_2}^p \circ \cdots \circ A_{0,d_n}^p((0,0))\right)
        = \lim_{n \to \infty} h_{d_1} \circ h_{d_2} \circ \cdots \circ h_{d_n}(0).
    \end{align*}
Similarly, the maps $A_{1,d}^{p}$ correspond to the alternating Lüroth system from \cite{KKK90,KKK91}. The collection of all signed Lüroth expansions is then obtained from all possible compositions of the Lüroth and alternating Lüroth systems, and the parameter $p \in (0,1)$ governs the weight that is put on each of them, or the probability with which the maps $A_{s,d}^p$ are chosen in such a composition.

For $J \subseteq \{0,1\} \times \mathbb{N}_{\geq 2}$, the projection onto the second coordinate of the limit set of the IFS $\{ A_{s,d}^p : (s,d) \in J \}$ contains precisely those numbers $x \in (0,1]$ that have a signed Lüroth expansion in which only digits $(s,d) \in J$ occur. Thereby, selecting different sets $J$ corresponds to placing different restrictions on the digits in the expansions. In this article we examine the geometry of the \emph{restricted digit sets} 
    \begin{align}\label{q:defF1d}
        \begin{aligned}
            F_{\!J}
            &= \{ x \in (0,1] : x \; \text{has a signed Lüroth expansion with all digits in} \; J\}\\
            &= \left\{ x \in (0,1] : \; \text{there exists} \;((s_n,d_n))_{n \in \mathbb{N}} \in J^\mathbb{N} \; \text{with} \; x = \pi_2 \left(\lim_{n \to \infty} A_{s_1,d_1}^p \circ A_{s_2, d_2}^p \circ \cdots \circ A_{s_n,d_n}^p((0,0))\right)\right\},
        \end{aligned}
    \end{align}
as well as the geometry of the self-affine sets $\mathcal{F}_{\!J}^p$. Note, $F_{\!J}$ does not depend on the parameter $p$, which is why we have omitted it from the notation. Similar to \cite[Theorem~4.3]{BF2023}, in \Cref{theorem:1DinfiniteJ} we obtain expressions for the Hausdorff, upper box-counting and packing dimensions of $F_{\!J}$. Moreover, we use results from \cite{rempe2016non} to obtain an expression for the Hausdorff dimension of non-autonomous variants of $F_J$, where the set $J$ describing the restriction can change at each time step. We use these results in tandem with the results of \cite{Mar54} to show the following.

\begin{theorem}\label{theorem:2DF}
    Let $I_0$ and $I_1$ denote two non-empty subsets of $\mathbb{N}_{\geq 2}$ and let $J = (\{0\} \times I_0)\cup (\{1\}\times I_1)$.  For $p \in (0,1)$,
        \begin{align}\label{eq:dimesion_of_F_J_p}
            \dim_{\mathcal H}(\mathcal F_{\!J}^p) \geq 1 + \inf\left\{ r\in (0,1] :  \left(\sum_{d_0\in I_0} \left(\frac{1}{d_0(d_0-1)}\right)^r\right)^p \left(\sum_{d_1\in I_1} \left(\frac{1}{d_1(d_1-1)}\right)^r\right)^{1-p} \leq  1\right\}.
        \end{align}
    Further, if $I_0 = I_1 = I \subseteq \mathbb{N}_{\geq 2}$ and $p \in (0,1)$, then 
        \begin{align}\label{eq:dimesion_of_F_J_p_2}
            \dim_{\mathcal H}(\mathcal F_{\!J}^p)
                = d(L_{s,d}^p \ \vert \ (s,d)\in J)
                = 1 + \inf\left\{ r\in (0,1] : \sum_{d\in I} \left(\frac1{d(d-1)}\right)^r \leq 1 \right\}.
        \end{align}
   In particular, if $I$ is finite, then $\dim_{\mathcal{H}}(\mathcal F_{\!J}^p) = \dim_{P}(\mathcal F_{\!J}^p) = \dim_{B}(\mathcal F_{\!J}^p) = d(L_{s,d}^p \ \vert \ (s,d)\in J)$.
\end{theorem}

\Cref{theorem:2DF} shows that our family of examples includes finitely and infinitely generated planar self-affine sets that are not irreducible, whose canonical projections to the $x$-coordinate are not necessarily exponentially separated, and for which the Hausdorff and affinity dimensions coincide; complementing the work of \cite{Fraser_2012,KR14,Jurga_dimension_spectrum,Rap_2023}. In the finite case this implies that the Hausdorff, packing, box-counting and affinity dimensions all coincide. \Cref{fig:the sets} illustrates approximations of sets $\mathcal F_{\!J}^p$ with $J=\{0,1\} \times I$ for various sets $I$.

\begin{figure}[t!]
    \centering
    \null\hfill
    \begin{subfigure}{0.285\linewidth}
    \centering
        \includegraphics[width = \linewidth]{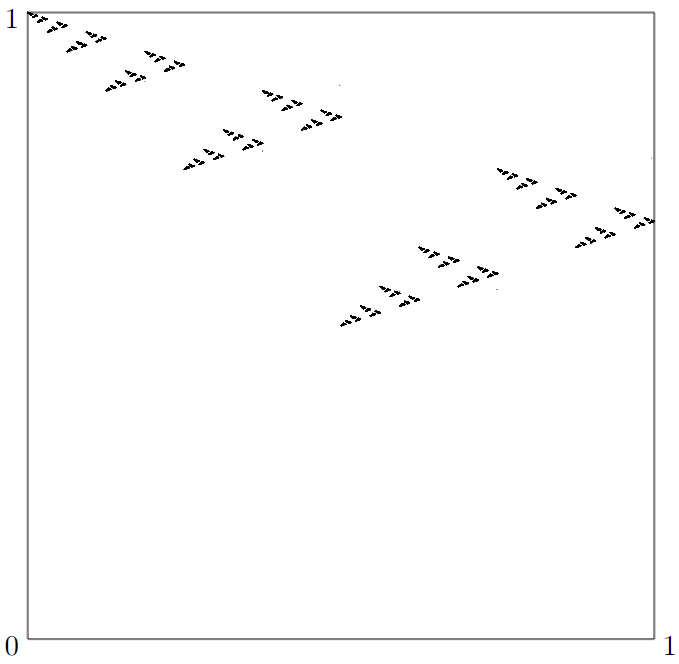}
        \caption{$J_1 = \{(0,2),(1,2)\}$}
    \end{subfigure}
    \hfill\null\hfill
    \begin{subfigure}{0.285\linewidth}
        \centering
        \includegraphics[width = \linewidth]{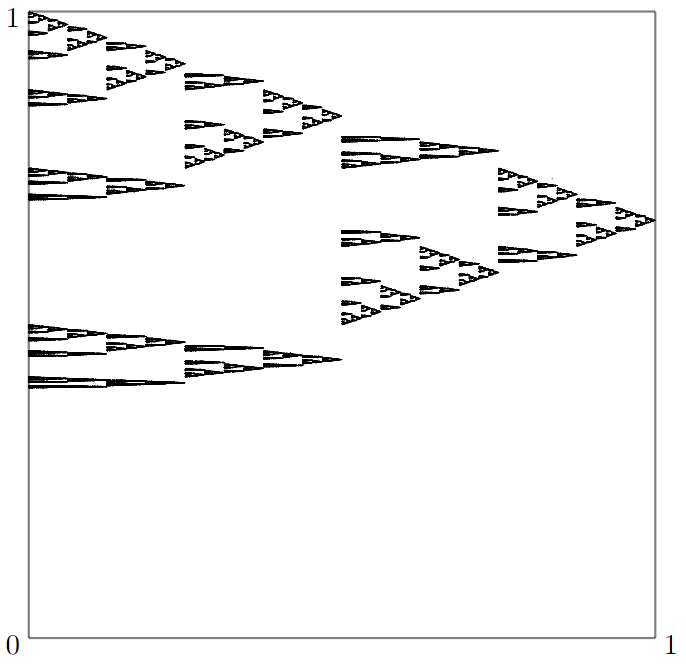}
        \caption{$J_2 = \{(0,2),(1,2),(0,3)\}$}
    \end{subfigure}
    \hfill\null\hfill
    \begin{subfigure}{0.285\linewidth}
        \centering
        \includegraphics[width = \linewidth]{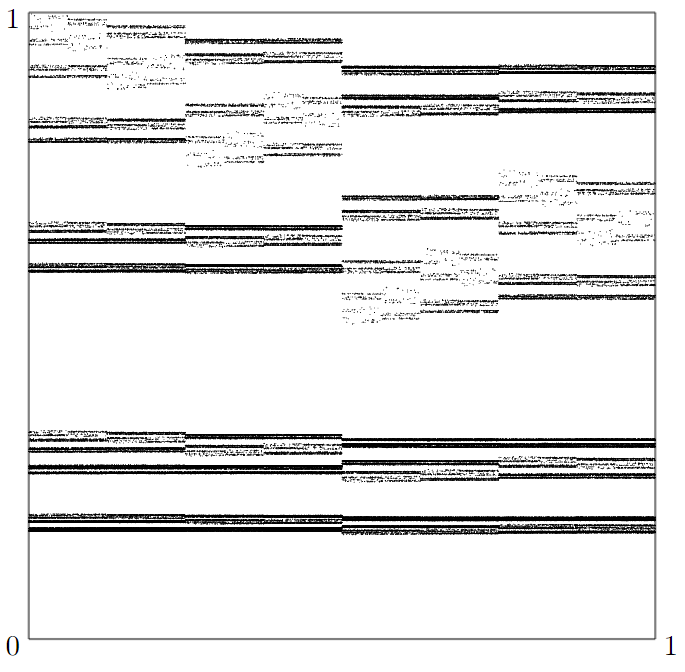}
        \caption{$J_3 = \{0,1\}\times \{2,4,6\}$}
    \end{subfigure}
    \hfill \null
    \caption{The self-affine sets $\mathcal F_{\!J_1}^p, \mathcal F_{\!J_2}^p$ and $\mathcal F_{\!J_3}^p$ for $p = \tfrac12$.}
    \label{fig:the sets}
\end{figure}

Families of finitely generated self-affine sets where the Hausdorff, packing, box-counting and affinity dimensions coincide were already shown to exist in \cite{MS19}. For planar self-affine sets generated by irreducible infinite IFSs it was shown in \cite{KR14,Jurga_dimension_spectrum} that the Hausdorff and affinity dimensions coincide. Our class of self-affine sets provides new examples where these equalities hold. Further, in \cite{Morris} it was proven that the modified affinity dimension from \cite{Fraser_2012} can be simplified when $F$ is the limit set of a finitely generated affine IFS for which the linear parts of the affine maps consist of diagonal and anti-diagonal matrices as long as $\dim_B \pi_1(F) = \dim_B \pi_2(F)$. %and the ROSC holds. 
On our way to proving \Cref{theorem:2DF}, we show that one can drop the condition $\dim_B \pi_1(F) = \dim_B \pi_2(F)$ at the price of having only diagonal matrices and still obtain the same simplification, see \Cref{theorem:FraserSimplified}.

In the above, we discussed dimension results of the sets $F_{\!J}$ and $\mathcal{F}_{\!J}^{p}$ for fixed sets $J \subseteq \{0,1\} \times \mathbb{N}_{\geq2}$. An interesting related question is, given a real number $y \in [0,2]$ can we find $J \subseteq \{0,1\} \times \mathbb{N}_{\geq2}$ such that the dimension of $\mathcal{F}_{\!J}^{p}$ equals $y$?  This question is related to the \emph{Texan conjecture} \cite{Hensley_1996,MU02}, which concerns the density of the dimensions of bounded type continued fraction sets in $[0, 1]$. Its resolution \cite{Kess_Zhu_2006} has generated a wealth of results and questions on the topological structure of the dimension spectrum of infinite IFSs. In \cite{Urbanski_dimension_spectrum} it was shown that the dimension spectra of conformal graph directed Markov systems are compact and perfect and that the IFS resulting from the complex continued fractions algorithm has full dimension spectrum. These results were built on in \cite{Jurga_dimension_spectrum}, where  examples of non-irreducible infinite IFSs consisting of affine maps whose dimension spectrum is neither compact nor perfect were given. Here we show that the self-affine sets $\mathcal F_{\!J}^p$ have full Hausdorff dimension spectrum.
    
    \begin{theorem}\label{thm:dimension_spectrum}
        For $p \in (0,1)$ we have $\{ \dim_{\mathcal{H}}(\mathcal F_{\!J}^p) : J \subseteq \{0,1\} \times \mathbb{N}_{\geq 2} \} 
        =[0,2]$.
    \end{theorem}

\subsection*{Outline} 

In \Cref{s:prel} we introduce some notation and recall the necessary preliminaries. \Cref{section:generaldiagonal} concerns the results on the affinity dimension of countable collections of diagonal $2 \times 2$~matrices and the proofs of \Cref{theorem:MorrisInfinite} and \Cref{prop:lower bound}. In \Cref{section:1D} we discuss the Hausdorff, upper box-counting and packing dimensions of the self-affine subsets $F_{\!J}$ of $\mathbb{R}$ and the Hausdorff dimension of certain non-autonomous versions of $F_{\!J}$. These results will then be used  in \Cref{s:theorem2} to prove \Cref{theorem:2DF,thm:dimension_spectrum}.

\section{\textbf{Preliminaries}}\label{s:prel}

\subsection{Symbolic dynamics}

An \emph{alphabet} $I$ is a countable set of symbols, which we call \emph{digits}, equipped with the discrete topology. A \emph{word} $u$ with digits in $I$ is a finite concatenation of digits $u = (i_{1}, \ldots, i_{n})$ for some $n \in \mathbb{N}$ and $i_{j} \in I$ for all $j \in \{1,2,\ldots,n\}$. We let $I^{m}$ denote the set of all words of length $m$ with digits in $I$ and set $I^{+} = \bigcup_{m \in \mathbb{N}} I^{m}$. Let $I^{\mathbb{N}} = \{ (i_{k})_{k \in \mathbb{N}} : i_{k} \in I \; \text{for all} \; k \in \mathbb{N} \}$ denote the set of all (one-sided) infinite sequences with elements in $I$ and endow $I^{\mathbb{N}}$ with the product topology. With this topology the space $I^{\mathbb{N}}$ is metrisable and in the case that $I$ is finite $I^{\mathbb{N}}$ is also compact. 
For $i,j \in \mathbb{N}$ with $i \leq j$ and $\omega = (\omega_k)_{k \in \mathbb{N}} \in I^{\mathbb{N}}$ we let $\omega_{[i,j]} = (\omega_i, \omega_{i+1}, \ldots, \omega_{j}) \in I^{j-i+1}$. We use the same notation if $v \in I^k$ and $1 \leq i \leq j \leq k$ for some $k\in \mathbb{N}$. 

\subsection{Self-similar sets, self-affine sets, and the open set condition}\label{s:sss}

Fix $D \in \mathbb{N}$ and let $X$ denote a non-empty compact subset of $\mathbb{R}^{D}$. If $I$ is a countable alphabet, a family $\Phi = \{\phi_i : i \in I \}$ of (non-trivial) contractions $\phi_{i} : X \to X$ is called an \emph{iterated function system} (\emph{IFS}). %To avoid trivial cases, we always assume that $\# I  \geq 2$. 
We call $\Phi$ a \emph{finite IFS} if $I$ is a finite alphabet and an \emph{infinite IFS} if $I$ is a countably infinite alphabet. For any $m\in \mathbb N$ and any finite word $u = (u_{1}, u_{2},\ldots, u_{m})\in I^{+}$ we let 
    \begin{align}\label{q:comps}
        \phi_{u}=\phi_{u_{1}} \circ \phi_{u_{2}} \circ \cdots \circ \phi_{u_{m}},
    \end{align}
and for $\omega = (\omega_k)_{k \in \mathbb{N}} \in I^{\mathbb{N}}$ we observe that $(\phi_{\omega_{[1,k]}}(X))_{k \in \mathbb{N}}$ forms a nested sequence of non-empty compact sets with decreasing diameters. By the Cantor Intersection Theorem, $\bigcap_{k \in  \mathbb{N}} \phi_{\omega_{[1,k]}}(X)$ is a singleton and we denote its only element by $\pi(\omega)$. We call the map $\pi : I^{\mathbb{N}} \to X$ the \emph{projection map} of $\Phi$ and refer to $\pi(I^{\mathbb{N}})$, the image of $I^{\mathbb{N}}$ under $\pi$, as the \emph{limit set} of $\Phi$. When $I$ is finite the natural action of $\Phi$ on the set of compact non-empty subsets of $X$, defined via $\Phi(A) = \bigcup_{i \in I} \phi_i(A)$, is a contraction with respect to the Hausdorff metric. By the Banach Contraction Mapping Principle there exists a unique non-empty set $E$ satisfying $\Phi(E)=E$. Moreover, in this setting, $E=\pi(I^{\mathbb{N}})$. 

Independent of $I$ being finite or countably infinite, if the contractions of $\Phi$ are all \textit{similarities}, that is, if for all $i \in I$ there exists $c_i \in (0,1)$ with $\lvert \phi_i(x) - \phi_i(y) \rvert = c_i \lvert x - y \rvert$ for all $x,y \in X$, then we call the limit set of $\Phi$ \emph{self-similar}. If the contractions of $\Phi$ are \textit{affine}, that is, if for each $i \in I$ there exists a $D \times D$ matrix $L_{i}$ whose singular values lie in  $(0,1)$ and a vector $v_i \in \mathbb{R}^{D}$ with $\phi_{i}((x_{1},\ldots,x_{D})) = (L_{i}(x_{1},\ldots,x_{D})^{\top} + v_i)^{\top}$, for all $(x_{1},\ldots,x_{D}) \in X$, then we call the limit set of $\Phi$ \emph{self-affine}.  

Two natural separation conditions we will use are the open set condition and the rectangular open set condition.  We say that $\Phi$ satisfies the \emph{open set condition (OSC)} if there exists a non-empty open subset $U$ of $X$ such that $\phi_{i}(U) \subseteq U$ and $\phi_{i}(U) \cap \phi_{j}(U) = \emptyset$ for all distinct $i,j \in I$. Such sets $U$ will be called \emph{feasible open sets for $\Phi$.} For self-affine sets we sometimes require a slightly stronger separation condition, namely that the OSC is satisfied with $U = (a_{1},b_{1}) \times \cdots \times (a_{D},b_{D})$, for some $a_{1}, \ldots, a_{D}, b_{1}, \ldots, b_{D} \in \mathbb{R}$ with $a_k<b_k$ for all $k\in\{1.\ldots,D\}$.  We refer to this latter separation condition as the \emph{rectangular open set condition (ROSC)}. 

The above represents the \emph{autonomous} setting; the contractions in \eqref{q:comps} are chosen from the same IFS at each time step. A more general setting is the \emph{non-automonous} setting, which is where the IFS is allowed to vary at each time step.  Formally, a \emph{non-automous self-similar iterated function system} (\emph{NSIFS}) consists of a sequence $\Phi=(\Phi^{(n)})_{n \in \mathbb{N}}$ of self-similar IFSs $\Phi^{(n)} = \{ \phi_{i}^{(n)} : i \in I^{(n)} \}$ defined on a common non-empty compact set $X \subseteq \mathbb{R}^{D}$.

As in the autonomous case we observe, for $(\omega_n)_{n \in \mathbb{N}} \in \prod_{n \in \mathbb{N}} I^{(n)}$, that $(\phi_{\omega_{1}}^{(1)} \circ \cdots \circ \phi_{\omega_{n}}^{(n)}(X))_{n \in \mathbb{N}}$ forms a nested sequence of non-empty compact sets with decreasing diameters.  Therefore $\bigcap_{n \in \mathbb{N}} \phi_{\omega_{1}}^{(1)} \circ \cdots \circ \phi_{\omega_{n}}^{(n)}(X)$ is a singleton. As above, let us denote the element of this singleton by $\pi(\omega)$.  We refer to the map $\pi : \prod_{n \in \mathbb{N}} I^{(n)} \to X$ as the \emph{projection map} of $\Phi$, and call the image $\pi(\prod_{n \in \mathbb{N}} I^{(n)})$ of $\prod_{n \in \mathbb{N}} I^{(n)}$ under $\pi$ the \emph{limit set} of $\Phi$. Further, we say that the NSIFS $\Phi$ satisfies the OSC if $\phi_{i}^{(n)}(\operatorname{int}(X)) \cap \phi_{j}^{(n)}(\operatorname{int}(X)) = \emptyset$ for all $n \in \mathbb{N}$ and all distinct $i,j \in I^{(n)}$.  

In order to obtain dimension estimates on the limit set of an NSIFS, we will assume the OSC and additionally the \emph{uniform contraction condition}. The latter means there exists an $\eta \in (0,1)$ such that for each $j \in \mathbb{N}$ we have $c_{\omega_{j}}^{(j)} c_{\omega_{j+1}}^{(j+1)} \cdots c_{\omega_{j+m}}^{(j+m)} \leq \eta^{m}$ for all sufficiently large $m \in \mathbb{N}$, where $\omega_{k} \in I^{(k)}$ and where $c_{\omega_{k}}^{(k)}$ denotes the contraction ratio of the similarity $\phi_{\omega_{k}}^{(k)}$ for each $k \in \mathbb{N}$. For further details on NSIFSs we refer the reader to \cite{rempe2016non}.

\subsection{Box-counting, Hausdorff, and (modified) affinity dimensions}

In this section we introduce the notions of dimension which we will mainly be concerned with, namely the box-counting, Hausdorff, affinity and modified affinity dimensions. We will also touch on the packing dimension, but since we do not use its definition directly we omit it. For more information on these notions of dimension we refer the reader to \cite{MR1449135,MR3236784}.

Fix $D\in \mathbb{N}$. The \emph{lower} and \emph{upper box-counting dimensions} of a bounded set $F \subseteq \mathbb{R}^{D}$ are defined by
    \begin{align*}
        \underline{\dim}_{B}(F) = \liminf_{\delta \to 0} \frac{\log(N_{\delta}(F))}{-\log(\delta)}
        \qquad \text{and} \qquad 
        \overline{\dim}_{B}(F) = \limsup_{\delta \to 0} \frac{\log(N_{\delta}(F))}{-\log(\delta)}
    \end{align*}
respectively, where $N_{\delta}(F)$ denotes the smallest cardinality of a $\delta$-cover of $F$, or alternatively, the number of closed squares in a $\delta$-mesh whose intersection with $F$ is non-empty.  When $\underline{\dim}_{B}(F)$ and $\overline{\dim}_{B}(F)$ are equal we refer to the common value as the \emph{box-counting dimension} of $F$ and denote it by $\dim_{B}(F)$.

Let $F$ be as above and let $s$ and $\delta$ denote two non-negative real numbers.  We define the \emph{$\delta$-approximate to the $s$-dimensional Hausdorff outer measure of} $F$ to be $\mathcal{H}_{\delta}^{s}(F) = \inf \left\{ \sum_{i} \diam(U_{i})^s : F \subseteq \bigcup_{i} U_{i} \; \text{and} \; 0 \leq \diam(U_{i}) < \delta \right\}$, where $\diam(U_{i})$ denotes the diameter of $U_{i}$.  The \emph{$s$-dimensional Hausdorff outer measure of} $F$ is given by $\mathcal{H}^{s}(F) = \lim_{\delta \to 0} \mathcal{H}_{\delta}^{s}(F)$ and the \emph{Hausdorff dimension} of $F$ is $\dim_{\mathcal{H}}(F) = \inf \{ s \geq 0 \colon \mathcal{H}^{s}(F) = 0 \}$, which coincides with the value $\sup \{ s \geq 0 \colon \mathcal{H}^{s}(F) = \infty \}$. Note that for a bounded set $F \subseteq \mathbb{R}^{D}$ these dimensions satisfy the relations $\dim_{\mathcal{H}}(F) \leq \underline{\dim}_{B}(F) \leq \overline{\dim}_{B}(F)$ and $\dim_{\mathcal{H}}(F) \leq \dim_{P}(F) \leq \overline{\dim}_{B}(F)$. However, in general, there is no relationship between the lower box-counting and packing dimensions of a given set.

Let $M_{D}(\mathbb{R})$ denote the collection of $D \times D$~matrices over $\mathbb{R}$. Given $L \in M_{D}(\mathbb{R})$ and $i \in \{1,2,\ldots, D \}$, we denote the $i$-th largest singular value of $L$, including multiplicities, by $\alpha_{i}(L)$ and we define the \emph{singular value function} $\varphi^{r}$ by
    \begin{align}\label{eq:singular-value-fcn}
        \varphi^{r}(L) = 
            \begin{cases}
                \alpha_{1}(L) \alpha_{2}(L) \cdots \alpha_{\lceil r \rceil-1}(L) (\alpha_{\lceil r \rceil}(L))^{r-\lceil r \rceil+1} & \text{if} \; r \in (0, D],\\[0.25em]
                \lvert \det(L) \rvert^{r/D} & \text{if} \; r > D,
            \end{cases}
    \end{align}
where $\lceil r\rceil = \min\{k\in\mathbb{Z} \,:\, k \geq r \}$. It is through this function that for finite alphabets $I$ the \emph{affinity dimension} $d( L_{i} \ \vert \ i \in I)$ of a collection of matrices $\{ L_i \}_{i \in I}$ was defined by Falconer in \cite{falconer_1988} by setting
    \begin{align*}
        d( L_{i} \ \vert \ i \in I) = \inf \left\{ r \in (0, D] \, : \, \sum_{m \in \mathbb{N}} \sum_{u \in I^{m}} \varphi^{r}(L_{u}) < \infty \right\}
    \end{align*}
where for $u = (u_{1}, \ldots, u_{m}) \in I^{m}$ with $m \in \mathbb{N}$ we set $L_{u} = L_{u_{1}}\cdots L_{u_{m-1}} L_{u_{m}}$.

The modified affinity dimension introduced in \cite{Fraser_2012} by Fraser for box-like self-affine sets is a variant of Falconer's affinity dimension that relies on knowledge of the dimensions of the projection of the given self-affine set $F$ onto the coordinate axes. Moreover, it is defined only when the ambient space is $\mathbb{R}^{2}$. Before defining the modified affinity dimension, we introduce some further notation.

Let $\Phi = \{ \phi_i : i \in I \}$ be a finite IFS containing affine maps $\phi_i : [0,1]^2 \to [0,1]^2$ defined, for $i \in I$ and $(w,x) \in \mathbb{R}^{2}$, by $\phi_i(w,x) = (L_i (w,x)^{\top} + v_i)^{\top}$ where $L_i \in M_2(\mathbb{R})$ is a diagonal matrix and $v_i \in \mathbb{R}^{2}$ is a translation vector. Assume that $\Phi$ satisfies the ROSC and let $F \subseteq \mathbb{R}^{2}$ denote the limit set of $\Phi$. Observe that, under our assumptions, the projections $\pi_{1}(F)$ and $\pi_{2}(F)$ are self-similar subsets of $\mathbb{R}$ generated by finite IFSs. Assuming that these systems satisfy the OSC, both $\dim_{B}(\pi_{1}(F))$ and $\dim_{B}(\pi_{2}(F))$ exist, see for instance \cite{Hut1981}. For $u \in I^+$ we define $\pi_{u} : \mathbb{R}^{2} \to \mathbb{R}$ by
    \begin{align}\label{q:piu}
        \pi_{u} =
            \begin{cases}
                \pi_{1} & \text{if} \; 
                \diam(\pi_{1}(\phi_{u}([0,1]^{2})) ) \geq \diam(\pi_{2}(\phi_{u}([0,1]^{2}))),\\[0.25em]
                \pi_{2} & \text{if} \; 
                \diam(\pi_{1}(\phi_{u}([0,1]^{2}))) < \diam(\pi_{2}(\phi_{u}([0,1]^{2}))),
            \end{cases}
    \end{align}
and set $r(u) = \dim_{B}(\pi_{u}(F))$.
For $r > 0$ and $u \in I^+$ the \emph{modified singular value function} $\varphi_{\text{mod}}^{r}$ of $L_{u}$ is defined by 
    \begin{align}\label{eq:mod-sing-value}
        \varphi_{\text{mod}}^{r}(L_{u}) 
        = \alpha_{1}(L_{u})^{r(u)} \alpha_{2}(L_{u})^{r-r(u)}.
    \end{align}
Note that these definitions are simplified slightly compared to the original definitions in \cite{Fraser_2012}, as we will only consider affine contractions with diagonal linear parts. As a consequence each box-like set in the present article will be of \textit{separated type}, meaning each contraction maps horizontal lines to horizontal lines.

In \cite{Fraser_2012} it was shown that for a finite IFS the \textit{modified pressure function} $P_{\text{mod}} \colon \mathbb{R}_{> 0} \to \mathbb{R}_{> 0}$ given by 
    \begin{align}\label{eq:pressure_mod_defn}
        P_{\text{mod}}(r) = \lim_{n \to \infty} \left(\sum_{u\in I^n} \varphi_{\text{mod}}^{r}(L_{u}) \right)^{1/n}
    \end{align}
is well defined and strictly decreasing in $r$.
Furthermore, it was shown that there exists a unique $t \in \mathbb{R}_{> 0}$, which we will refer to as the \emph{modified affinity dimension} of $F$, satisfying $P_{\text{mod}}(t) = 1$, and that under the given assumptions we have $\dim_{B}(F) = \dim_{P}(F) = t$.

\section{\textbf{Affinity dimensions for infinite affine IFSs with diagonal linear parts}} 
\label{section:generaldiagonal}

In this section we prove \Cref{theorem:MorrisInfinite,prop:lower bound}.  Following this we show, in \Cref{theorem:FraserSimplified}, that in our setting the modified affinity dimension from \cite{Fraser_2012} can be simplified. We begin with the proof of \Cref{theorem:MorrisInfinite} where we utilise ideas from \cite{Morris}.

\begin{proof}[Proof of \Cref{theorem:MorrisInfinite}]
    Let $I$ be a countable alphabet and $\{L_i : i\in I\}$ a collection of diagonal $2\times 2$~matrices, as given in \eqref{eq:Li}, with $\sup_{i\in I} \max\{ \lvert a_i \rvert, \lvert b_i \rvert \} < 1$. Note that the singular value function $\varphi^r(L)$ of a matrix $L$, as defined in \eqref{eq:singular-value-fcn}, is non-negative, strictly decreasing and continuous in $r$, so by the root test,
        \begin{align*} 
            d(L_i \ \vert \ i\in I) =  \inf\left\{ r > 0  : \limsup_{m \to \infty} \left(\sum_{u\in I^m} \varphi^r(L_u)\right)^{1/m} \leq 1\right\}.
        \end{align*}
    For a diagonal matrix $L=\begin{pmatrix} a & 0\\ 0 & b \end{pmatrix}$ and any $r>0$ define the matrix
        \begin{align*}
            L^{(r)} =
                \begin{pmatrix}
                    a^{(r)} & 0\\
                    0 & b^{(r)}
                \end{pmatrix}
            = \begin{cases}
                    \begin{pmatrix}
                        \lvert a \rvert^r & 0 \\
                        0 & \lvert b \rvert^r
                        \end{pmatrix} & \text{if} \; 0 < r\leq 1\\[1.25em]
                    \begin{pmatrix}
                    \lvert a \rvert \cdot \lvert b \rvert^{r-1} & 0 \\
                    0 & \lvert b \rvert \cdot \lvert a \rvert^{r-1}
                    \end{pmatrix} & \text{if} \; 1<r\leq 2\\[1.25em]
                    \begin{pmatrix}
                    \lvert a b\rvert^{r/2} & 0 \\
                    0 & \lvert a b \rvert^{r/2}
                    \end{pmatrix} & \text{if} \; 2<r.
            \end{cases}
        \end{align*}
    Next, we show $\varphi^r(L) = \lVert L^{(r)} \rVert$ for each $r>0$, where $\lVert \cdot \rVert$ denotes the operator norm on $M_2(\mathbb{R})$:
        \begin{itemize}
            \item For $0 < r\leq 1$ we have $\varphi^r(L) = \alpha_1(L)^r = \max\{\lvert a \rvert^r, \lvert b \rvert^r\} = \lVert L^{(r)} \rVert$.\\[-0.75em]
            \item For $1 < r \leq 2$, since $\alpha_1(L) \geq \alpha_2(L)$ by definition, we have $\tfrac{\alpha_1(L)}{\alpha_2(L)} \geq 1$ and hence $(\tfrac{\alpha_1(L)}{\alpha_2(L)} )^{r-1} \leq \tfrac{\alpha_1(L)}{\alpha_2(L)}$, or equivalently $\alpha_1(L)\cdot \alpha_2(L)^{r-1} \geq \alpha_2(L) \cdot \alpha_1(L)^{r-1}$. This implies that $\varphi^r(L) = \max\{\lvert a \rvert \cdot \lvert b \rvert^{r-1}, \lvert b \rvert \cdot \lvert a \rvert^{r-1}\} = \lVert L^{(r)} \rVert$.\\[-0.75em]
            \item For $r>2$ we have $\varphi^r(L) = \lvert \det(L)\rvert^{r/2} = \lvert a b\rvert^{r/2} = \lVert L^{(r)} \rVert$.
        \end{itemize}
    Observe furthermore that $(LK)^{(r)} = L^{(r)}K^{(r)}$ for diagonal matrices $L$ and $K$. Therefore, if we return to our collection of matrices $\{ L_i : i \in I\}$, we can unambiguously set $L_u^{(r)} = L_{i_1}^{(r)} \cdots L_{i_m}^{(r)} = (L_{i_1} \cdots L_{i_m})^{(r)}$ for each $u = (i_1,\ldots,i_m)\in I^m$. In particular, it follows that 
        \begin{align}\label{eq:svf-norm}
            \varphi^r(L_{u}) = \lVert L_{u}^{(r)} \rVert\quad\text{for any}\ u \in I^m\ \text{and}\ r>0.
        \end{align}
    Suppose that for some $r>0$ the series $\sum_{i\in I} L_i^{(r)}$ does not converge in $(M_{2}(\mathbb{R}),\|\cdot\|)$. In other words, suppose that $P_I(r) = \max\{ \sum_{i\in I} a_i^{(r)}, \sum_{i\in I} b_i^{(r)}\} = \infty$. In this case, \eqref{eq:svf-norm} implies for $m\in \mathbb{N}$ that
        \begin{align*}
            \sum_{u\in I^m} \varphi^r(L_{u})
            = \sum_{u\in I^m} \lVert L_{u}^{(r)}\rVert
            &= \sum_{i_1,\ldots,i_m\in I} \max\left\{ a_{i_1}^{(r)}\cdots a_{i_m}^{(r)}, b_{i_1}^{(r)}\cdots b_{i_m}^{(r)} \right\}\\ 
            &\geq \max\left\{ \sum_{i_1,\ldots,i_m\in I} a_{i_1}^{(r)}\cdots a_{i_m}^{(r)}, \sum_{i_1,\ldots,i_m\in I} b_{i_1}^{(r)}\cdots b_{i_m}^{(r)} \right\}
            = \max\left\{\sum_{i\in I} a_i^{(r)}, \sum_{i\in I} b_i^{(r)} \right\}^m
            = \infty.
        \end{align*}
    Hence $\limsup_{m \to \infty} \left(\sum_{u\in I^m} \varphi^r(L_u)\right)^{1/m} = P_I(r)$.
    If on the other hand the series $\sum_{i\in I} L_i^{(r)}$ converges, then
        \begin{align*}
            \left\lVert \sum_{i\in I} L_i^{(r)}  \right\rVert = \max\left\{ \sum_{i\in I} a_i^{(r)}, \sum_{i\in I} b_i^{(r)}\right\} = P_I(r) < \infty.
        \end{align*}
    For $L \in M_2(\mathbb{R})$ let $\lvert L \rvert$ denote the sum of the absolute values of the components of $L$, and observe that $\lvert \cdot \rvert$ is a norm on $M_2(\mathbb{R})$. If $L$ and $K$ are non-negative diagonal matrices we have $\lvert L+K \rvert = \lvert L \rvert + \lvert K \rvert$ and so by the continuity of norms we have $\lvert \sum_{i=1}^\infty K_i \rvert = \sum_{i=1}^\infty \lvert K_i \rvert$ for any sequence $(K_i)_i$ of non-negative diagonal matrices. Note, for a diagonal matrix 
        \begin{align*}
            L=\begin{pmatrix} a & 0\\ 0 & b \end{pmatrix}
        \end{align*}    
    and $m\in \mathbb{N}$ we have $\lVert L \rVert = \max\{ \lvert a \rvert, \lvert b \rvert \} = \max\{ \lvert a \rvert^m, \lvert b \rvert^m \}^{1/m} = \lVert L^m \rVert^{1/m}$. Since $M_2(\mathbb{R})$ is a finite-dimensional vector space, the norms $\lVert \cdot \rVert$ and $\lvert \cdot \rvert$ are equivalent. Therefore, there exists $c>0$ such that $c^{-1}|L|\leq \|L\|\leq c|L|$ for any $L\in M_2(\mathbb{R})$. Combining the above yields the following for $m\in \mathbb{N}$;
        \begin{align}\label{eq:svf-c-upper}
            \sum_{u\in I^m} \varphi^r(L_{u})  
            =   \sum_{u\in I^m} \left\lVert L_{u}^{(r)}  \right\rVert 
            \leq \sum_{u\in I^m} c\left\lvert L_{u}^{(r)}  \right\rvert 
            = c\left\lvert \sum_{u\in I^m} L_{u}^{(r)} \right\rvert
            \leq c^{2} \left\lVert \sum_{u\in I^m} L_{u}^{(r)} \right\rVert
            = c^{2}\left\lVert \left( \sum_{i\in I} L_i^{(r)}\right)^m \right\rVert
            = c^{2}\left\lVert \sum_{i\in I} L_i^{(r)} \right\rVert^m.
        \end{align}
        Likewise
        \begin{align}\label{eq:svf-c-lower}
            \sum_{u\in I^m} \varphi^r(L_{u})  
            \geq c^{-2}\left\lVert \sum_{i\in I} L_i^{(r)} \right\rVert^m.
        \end{align}     
        Equations \eqref{eq:svf-c-upper} and \eqref{eq:svf-c-lower} together imply
        \[
            \limsup_{m \to \infty} \left( \sum_{u\in I^m} \varphi^r(L_{u}) \right)^{1/m} 
            = \left\lVert \sum_{i\in I} L_i^{(r)} \right\rVert = P_I(r).
            \qedhere
        \]
\end{proof}

We now prove \Cref{prop:lower bound}. For \eqref{item:prop:lower_bound_1}, that is, to show that the affinity dimension is always an upper bound for the Hausdorff dimension, we follow steps similar to those in the proof for \cite[Proposition~5.1]{falconer_1988}.

\begin{proof}[Proof of \Cref{prop:lower bound}]
    For Part \eqref{item:prop:lower_bound_1}, since $F \subseteq \mathbb{R}^2$, we naturally have $\dim_{\mathcal{H}}(F)\leq2$.  Therefore it is sufficient to show that $\dim_{\mathcal{H}}(F) \leq d(L_i \ \vert \ i \in I )$. To this end, let $\delta > 0$ be given. By the assumption that $\sup_{i\in I} \max\{ \lvert a_i \rvert, \lvert b_i \rvert \} < 1$ there exists some integer $k_\delta$ such that for all sequences $u\in I^{k_\delta}$ we have $\alpha_1(L_{u}),\alpha_2(L_{u}) < \delta$. Now take any $m\geq k_\delta$ and recall that $F\subseteq \bigcup_{u \in I^m} A_{u}([0,1]^2)$. For each $u\in I^m$, $A_{u}([0,1]^2)$ is a rectangle with side lengths $\alpha_1(L_{u})$ and $\alpha_2(L_{u})$. This rectangle can be covered by $\lceil \alpha_1(L_{u})/\alpha_2(L_{u})\rceil$ squares of side length $\alpha_2(L_{u})$, and hence also by this many circles of diameter $\sqrt{2} \alpha_2(L_{u})$. Note that, since $\alpha_1(L_{u})/\alpha_2(L_{u})\geq 1$, we have $\lceil \alpha_1(L_{u})/\alpha_2(L_{u})\rceil \leq 2 \alpha_1(L_{u})/\alpha_2(L_{u})$. For every $0 <  r \leq 2$ we have
        \begin{align*}
            \mathcal H^r_{\sqrt{2}\delta}(F)
            \leq \sum_{u\in I^m} 2\frac{\alpha_1(L_{u})}{\alpha_2(L_{u})} \left( \sqrt{2} \alpha_2(L_{u})\right)^r
            = 2\left(\sqrt{2}\right)^r \sum_{u \in I^m} \alpha_1(L_{u}) \alpha_2(L_{u})^{r-1} \leq 4 \sum_{u \in I^m} \varphi^r(L_{u}).
        \end{align*}
    Since this holds for all $m\geq k_\delta$, and since $k_\delta$ diverges to infinity as $\delta$ tends to zero,
        \begin{align*}
            0 
            \leq \mathcal H^r(F)
            \leq 4 \limsup_{m \to \infty} \sum_{u \in I^m} \varphi^r(L_{u}),
        \end{align*}
    for each $0 < r \leq 2$. Now for any $r$ satisfying $\sum_{m=1}^\infty \sum_{u \in I^m} \varphi^r(L_{u}) < \infty$ we have $\limsup_{m \to \infty} \sum_{u \in I^m} \varphi^r(L_{u}) < \infty$ and so $\mathcal H^r(F) < \infty$. Thus, $\dim_{\mathcal{H}}(F) = \inf\{ r \geq 0 : \mathcal H^r(F) < \infty \} \leq \inf\{ r > 0 :  \sum_{m=1}^\infty \sum_{u \in I^m} \varphi^r(L_{u}) < \infty\} = d(L_i \ \vert \ i \in I)$. This concludes the proof of \Cref{prop:lower bound} \eqref{item:prop:lower_bound_1}.

    For Part (\ref{item:prop:lower_bound_2}) observe that $P_I(r) = \sup_{I'\subseteq I \; \text{finite}} P_{I'}(r)$ for all $r > 0$ with $P_I(r)$ as defined in \eqref{eq:pressure}. By assumption, for a finite subset $I_{2}$ with $I_1 \subseteq I_2 \subseteq I$, we have $d(L_i \ \vert \ i\in I_2) = \dim_{B}(F_{I_2}) \le 2$. By \Cref{theorem:MorrisInfinite}, $P_{I_2}(r+\varepsilon)\leq 1$ for all $\varepsilon>0$ and all $r \geq 2$, which implies $d(L_i \ \vert \ i\in I) = \inf\{r > 0 : P_{I}(r)\leq 1\} \le 2$.

    Next, we show that
        \begin{align}\label{q:finitesubsets}
            d(L_i \ \vert \ i\in I) = \sup_{I'\subseteq I \; \text{finite}} d(L_i \ \vert \ i\in I'),
        \end{align} 
    from which we will conclude the required result. To this end, observe that each of the series in the definition of $P_I$ has positive terms, and thus, $P_{I_1}(r) \leq P_{I_2}(r)$ for $r > 0$ and $I_1 \subseteq I_2 \subseteq I$. Therefore, 
        \begin{align*}
            d(L_i \ \vert \ i \in I)
            = \inf\left\{r > 0 : \sup_{I'\subseteq I \; \text{finite}} P_{I'}(r)\leq 1\right\}
            = \inf\{r > 0 : P_{I'}(r)\leq 1 \; \text{for all finite} \; I'\subseteq I\}.
        \end{align*}
    Write $Z = \sup_{I'\subseteq I \; \text{finite}} d(L_i \ \vert \ i \in I') = \sup_{ I' \subseteq I \; \text{finite}} \inf\{r > 0 : P_{I'}(r)\leq 1\}$. For each finite subset $I'\subseteq I$ we have $Z \geq \inf\{r > 0 : P_{I'}(r) \leq 1\}$, so since $P_{I'}(r)$ is strictly decreasing in $r$ we also have $P_{I'}(Z) \leq 1$. As this holds for all finite $I' \subseteq I$, it follows that $d(L_i \ \vert \ i\in I) \leq Z$. Further, for each $r > d(L_i \ \vert \ i\in I)$ we have $P_{I'}(r) \leq 1$ for each finite $I'\subseteq I$ and hence $r\geq Z$, from which we conclude that $d(L_i \ \vert \ i \in I) = Z$.   
    
    For each $I_{1}\subseteq I$ we have $F_{I_{1}} \subseteq F$, where $F_{I_{1}}$ is as in our hypotheses of \Cref{prop:lower bound}\eqref{item:prop:lower_bound_2}.
    By the monotonicity of both the lower box-counting and affinity dimensions, and by \eqref{q:finitesubsets}, we have that
        \begin{align*}
            \underline{\dim}_{B}(F)
            \geq  \! \sup_{I_2\subseteq I \; \text{finite}} \!\underline{\dim}_{B}(F_{\!I_2})
            =  \! \sup_{I_1\subseteq I_2\subseteq I \; \text{finite}} \! \dim_{B}(F_{\!I_2})
            = \! \sup_{I_1\subseteq I_2\subseteq I \; \text{finite}} \! d(L_i \ \vert \ i \in I_2)
            = \! \sup_{I_2 \subseteq I \; \text{finite}} \! d(L_i \ \vert \ i \in I_2) = d(L_i \ \vert \ i \in I).
        \end{align*}
    Since the affine maps we consider are bi-Lipschitz, it follows from  \cite[Theorem~3.1]{mauldin1996dimensions} that 
        \begin{align*}
            \dim_{P}(F) = \overline{\dim}_{B}(F) = \dim_{P} (\overline F) = \overline{\dim}_{B}(\overline F).
        \end{align*}        
    Thus, under our assumptions,
    $d(L_i \ \vert \ i\in I) \leq \underline{\dim}_{B}(F) \leq \overline{\dim}_{B}(F) = \dim_{P}(F)$.
\end{proof}

\begin{remark}
    In \cite[Theorem~5.4]{falconer_1988} and \cite[Theorem~2.4]{Fraser_2012} it is shown that the affinity and modified affinity dimensions of a finite affine IFS is an upper bound for the upper box-counting dimension of the associated self-affine set.  This result relies on the fact that the singular values of the affine maps in a finite IFS are uniformly bounded from below by a positive constant. Such a lower bound on the singular values does not exist in general for infinite IFSs. Thus, the proofs of the aforementioned theorems do not naturally generalise to the case of infinite affine IFSs.
\end{remark}

In \cite{Fraser_2012} it was shown that for a class of finitely generated planar box-like self-affine sets the box-counting and packing dimensions are bounded above by (and when the ROSC is satisfied equal to) the modified affinity dimension, which is the unique $t \in \mathbb{R}_{> 0}$ solving $P_{\text{mod}}(t) = 1$, see \eqref{eq:pressure_mod_defn}. In \cite[Proposition~5]{Morris}, for a similar class of finitely generated planar box-like self-affine sets $F$, and under the assumption that $\dim_B \pi_1(F) = \dim_B \pi_2(F)$, a simple expression for $P_{\text{mod}}$ was obtained. Next, we show that in the case where we have only diagonal matrices the same simplification of $P_{\text{mod}}$ can be obtained no matter the values of $\dim_B \pi_1(F)$ and $\dim_B \pi_2(F)$.

\begin{prop} \label{theorem:FraserSimplified}
    Let $I$ be a finite alphabet and let $F$ be the limit set of an IFS $\{A_i : i\in I\}$ where each $A_i$ is an affine contraction on $[0,1]^2$ with linear part $L_{i} \in M_{2}(\mathbb{R})$ as given in \eqref{eq:Li}. Set $r_1 = \dim_B (\pi_1(F))$ and $r_2 = \dim_B (\pi_2(F))$ and assume $\{A_i : i\in I\}$ satisfies the ROSC. Under these assumptions, $\dim_B (F) = \dim_P (F) = t$, where $t \in \mathbb{R}_{> 0}$ is the unique solution to
        \begin{align*}
            \max\left\{\sum_{i\in I} \lvert a_i \rvert^{r_1} \lvert b_i \rvert^{t - r_1}, \sum_{i\in I} \lvert b_i \rvert^{r_2} \lvert a_i \rvert^{t - r_2} \right\}
            = 1.
        \end{align*}
\end{prop}

\begin{proof}
   Let $u = (i_1,\ldots,i_m)\in I^m$ for some $m\in \mathbb{N}$. The singular values of $L_u$ are $a(u) = \prod_{k=1}^m \lvert a_{i_k} \rvert$ and $b(u) = \prod_{k=1}^m \lvert b_{i_k} \rvert$. Recall that $r(u) = \dim_B (\pi_u (F))$, where $\pi_u$ is as in \eqref{q:piu}. Since for $j\in \{1,2\}$, $r(u)$ equals $r_j$ when $\alpha_1(L_u)$ corresponds to the contraction in the $j$-th coordinate, we obtain for each $r > 0$ that, with $\varphi^r_{\text{mod}}$ as in \eqref{eq:mod-sing-value},
    \begin{align*}
        \varphi^r_{\text{mod}}(L_u) = \begin{cases}
            a(u)^{r_1} b(u)^{r-r_1}& \text{if} \; a(u)\geq b(u),\\[0.25em]
            b(u)^{r_2} a(u)^{r-r_2}& \text{if} \; a(u) < b(u),
        \end{cases}
        \qquad \text{and set} \qquad
        \mathfrak{L}_{u}^{(r)} = \begin{pmatrix}
            \lvert a(u) \rvert^{r_1} \lvert b(u) \rvert^{r-r_1} & 0\\[0.25em]
            0 & \lvert b(u) \rvert^{r_2} \lvert a(u) \rvert^{r-r_2}
        \end{pmatrix}.
    \end{align*}
    Since $P_{\text{mod}}(\dim_B F) = 1$ by \cite{Fraser_2012} and $\dim_B F \leq r_1+r_2$, we need only consider the case $r \in (0,r_1+r_2]$.
    For $r > 0$,
        \begin{align*}
            \frac{a(u)^{r_1} b(u)^{r-r_1}}{b(u)^{r_2} a(u)^{r-r_2}} = \left(\frac{a(u)}{b(u)}\right)^{r_1+r_2-r}\hspace{-2.25em},
        \end{align*}
    and so for $r\in (0,r_1+r_2]$ we have $\frac{a(u)^{r_1} b(u)^{r-r_1}}{b(u)^{r_2} a(u)^{r-r_2}} \geq 1 $ when $a(u) \geq b(u)$, and $\frac{a(u)^{r_1} b(u)^{r-r_1}}{b(u)^{r_2} a(u)^{r-r_2}} < 1 $ otherwise. Thus,
        \begin{align*}
            \varphi^r_{\text{mod}}(L_u)
            = \max\left\{ a(u)^{r_1} b(u)^{r-r_1},\ b(u)^{r_2} a(u)^{r-r_2}\right\}
            = \lVert \mathfrak{L}_u^{(r)} \rVert,
        \end{align*}
    for $r\in (0,r_1+r_2]$. Following the same steps as in the proof of \Cref{theorem:MorrisInfinite},
        \begin{align*}
            \max\left\{\sum_{i\in I} \lvert a_i \rvert^{r_1} \lvert b_i \rvert^{r - r_1},\ \sum_{i\in I} \lvert b_i \rvert^{r_2} \lvert a_i \rvert^{r - r_2} \right\}
            =  \left\lvert \sum_{i\in I} \mathfrak{L}_i^{(r)} \right\rvert
            = \limsup_{m \to \infty} \left( \sum_{u \in I^m} \varphi_{\text{mod}}^r(L_u)\right)^{1/m}
            = P_{\text{mod}}(r),
        \end{align*}
    for each $r \in (0,r_1+r_2]$. The required result now follows from an application of \cite[Theorem~2.4]{Fraser_2012}.
\end{proof}

\section{\textbf{Box-counting and Hausdorff dimensions of \texorpdfstring{$F_{\!J}$}{F underscore J}}}
\label{section:1D}

Here we collect and develop results which allow us to compute the box-counting and Hausdorff dimensions of the sets $F_{\!J}$ from \eqref{q:defF1d}. These results are utilised in our proofs of \Cref{theorem:2DF,thm:dimension_spectrum}. Recall for $(s,d)\in \{0,1\}\times \mathbb{N}_{\geq 2}$ the definition of the maps $A_{s,d}^p:[0,1]^2 \to [0,1]^2$ from \eqref{q:Land}. We will be interested in their second coordinates, which are the maps $\phi_{s,d} : [0,1] \to [1/d,1/(d-1)]$ given by
 \begin{align*}
        \phi_{s,d}(x) = \frac{(-1)^{s}x}{d(d-1)} + \frac{1}{d-s}.
    \end{align*}
(Note that the maps $\phi_{0,d}$ correspond to the maps $h_d$ from the introduction.) For each $J \subseteq \{0,1\}\times \mathbb{N}_{\geq 2}$, the IFS $\Phi = \{\phi_{s,d} : (s,d)\in J\}$ on $[0,1]$ consists of similarities and its limit set is precisely the set $F_{\! J}$. It is a consequence of \cite{KM22b} that each number in $(\tfrac12,1]$ has at least one signed Lüroth expansion for which the corresponding digit pairs $(s_n,d_n)\in \{0,1\}\times \mathbb{N}_{\geq 2}$ satisfy $d_n = 2$ for all $n\in \mathbb{N}$. Consequently, if $(0,2),(1,2) \in J$, then the restricted digit set corresponding to $J$ contains $(\tfrac12,1]$, yielding 
    \begin{align}\label{eq:dimension_1D_0_2_and_1_2}
        \dim_{\mathcal{H}}(F_{\!J}) = \dim_{B}(F_{\!J}) = 1.
    \end{align}
In all other cases we split our analysis of $F_{\!J}$ into the cases where $J$ is finite and where $J$ is countably infinite.

\subsection{Restricted digit sets with finite alphabets for signed Lüroth expansions}

Throughout this section we assume that $J$ is a finite subset of $\{0,1\}\times \mathbb{N}_{\geq 2}$ and that $\{ (0,2), (1,2)\} \not\subseteq J$. Since the set $F_{\!J}$ is self-similar it follows from \cite[Theorem~5.3(1)]{Hut1981} that, in case the IFS $\Phi$ satisfies the OSC, the Hausdorff and box-counting dimensions of $F_{\!J}$ equal the unique $r \in \mathbb{R}_{> 0}$ satisfying
    \begin{align}\label{eq:theorem:1DfiniteJ}
        \sum_{(s,d) \in J} \left(\frac{1}{d(d-1)}\right)^{r} = 1.
    \end{align}
In particular, one can show that the OSC is satisfied when $(0,2),(1,2) \not \in J$ with feasible open set $(0,\tfrac{1}{2})$. Below we discuss examples of sets $J$ containing just one of $(0,2)$ or $(1,2)$ where the OSC is satisfied, \Cref{ex:OSCsatisfied}, and where the OSC is not satisfied, \Cref{ex:OSCnotsatisfied}.

\begin{ex} \label{ex:OSCsatisfied}
    Let $d\in \mathbb{N}_{\geq 3}$ and consider the set $J = \{(0,2),(0,d),(1,d)\}$. The IFS $\{\phi_{s,d} : (s,d)\in J\}$ satisfies the OSC with feasible open set $U = \bigcup_{k=0}^\infty \phi_{0,2}^k((\tfrac1d,\tfrac1{d-1}))$, but neither $(0,1)$ nor $(0,1/2)$ are feasible open sets. Here $\phi_{0,2}^0$ is defined to be the identity. To see that $U$ is a feasible open set for the OSC, we observe that by construction $\phi_{0,2}(U) \subseteq U$, and that $\phi_{s,d}(U)\subseteq \phi_{s,d}((0,1)) = (\tfrac1d,\tfrac1{d-1}) \subseteq U$ for $s\in \{0,1\}$. It remains to show that $\phi_{0,2}(U)$, $\phi_{0,d}(U)$ and $\phi_{1,d}(U)$ are pairwise disjoint. Since $\phi_{0,2}(U) \subseteq (\tfrac12,1)$, and since $\phi_{0,d}(U)$ and $\phi_{1,d}(U)$ are subsets of $(\tfrac1d,\tfrac1{d-1})\subseteq (0,\tfrac12)$, it suffices to verify that $\phi_{0,d}(U)\cap \phi_{1,d}(U) = \emptyset$. To this end, we define $U_1 = U\cap [0,\tfrac12] = (\tfrac1d,\tfrac1{d-1})$ and $U_2 = U\cap [\tfrac12,1] = \bigcup_{k=1}^\infty \phi_{0,2}^k((\tfrac1d,\tfrac1{d-1}))$. Observe that the injective maps $\phi_{0,d}$ and $\phi_{1,d}$ satisfy $\phi_{0,d}(x) = \phi_{1,d}(y)$ if and only if $y = 1-x$, in which case exactly one of $x$ and $y$ is an element of $U_1$ and exactly one of $x$ and $y$ is an element of $U_2$. As such, $\phi_{0,d}(U)\cap \phi_{1,d}(U) = \emptyset$ if and only if $U_2 \cap (1-U_1) = \emptyset$, where by $1-U_1$ we mean the open interval $(1-\tfrac{1}{d-1}, 1 - \tfrac{1}{d})$. Recalling that $\phi_{0,2}(x) = \tfrac12 x + \tfrac12$ for $x\in [0,1]$, we have $\phi_{0,2}^k(x) =  \frac1{2^k} x + \sum_{j=1}^k \frac{1}{2^j}$ for $k\in \mathbb{N}$. Thus, for $k \in \mathbb{N}$, we have that $\phi_{0,2}^{k}(\tfrac{1}{d-1}) < \phi_{0,2}^{k+1}(\tfrac{1}{d})$ if and only if $d^{2} - 2d - 1 > 0$, but this latter inequality holds since $d \geq 3$. This in tandem with the fact that $\phi_{0,2}$ is strictly increasing implies that, for $k, l \in \mathbb{N}$ with $k > l$, the open intervals $\phi_{0,2}^k((\tfrac1d,\tfrac1{d-1}))$ and $\phi_{0,2}^l((\tfrac1d,\tfrac1{d-1}))$ are disjoint, and if $x \in \phi_{0,2}^k((\tfrac1d,\tfrac1{d-1}))$ and $y \in \phi_{0,2}^l((\tfrac1d,\tfrac1{d-1}))$, then $x > y$. It remains to show that $1-U_1$ lies strictly in between two such consecutive intervals.

    For this, set $k = \lceil \log_2(d-1)\rceil - 1 \in\mathbb{N}$, and note that for this $k$ we have $\frac{1}{2^{k+1}} \leq \frac{1}{2^{k+1}} \frac{2^{k+1} + 1}{d} = (1+\frac{1}{2^{k+1}})\frac1d$, and hence $1-\frac1d \leq  \frac{1}{2^{k+1} d} + 1 - \frac{1}{2^{k+1}}$. Since $d-1$ is an integer, we have $k = \lceil \log_2(d-2 + 1)\rceil - 1 = \lfloor \log_2(d-2)\rfloor$ and so we also have $2^k\leq d-2$. Equivalently, we have $\frac1{d-1} \leq \frac{1}{2^k+1} = \frac{2^k}{2^k(2^k+1)} = \frac{1}{2^k}(\frac1{2^k}+1)^{-1}$, or $(\frac1{2^k}+1)\frac{1}{d-1} \leq \frac{1}{2^k}$. This yields
        \begin{align*}
            \phi_{0,2}^k\left(\frac{1}{d-1}\right)
            = \frac{1}{2^k(d-1)} + 1 - \frac{1}{2^k} \leq 1 - \frac{1}{d-1} < 1 - \frac{1}{d} \leq \frac{1}{2^{k+1} d} + 1 - \frac{1}{2^{k+1}}
            = \phi_{0,2}^{k+1}\left(\frac{1}{d}\right),
        \end{align*}
    from which we conclude that $1-U_1$ lies strictly between $\phi_{0,2}^k((\tfrac1d, \tfrac1{d-1}))$ and $\phi_{0,2}^{k+1}((\tfrac1d, \tfrac1{d-1}))$ and hence $(1-U_1) \cap U_2 = \emptyset$.
    \end{ex}

\begin{ex}\label{ex:OSCnotsatisfied}
    By \cite[Theorem~5.3(1)]{Hut1981}, if $J$ is such that $\sum_{(s,d)\in J} \frac{1}{d(d-1)} > 1$, then since $F_{\!J} \subseteq [0,1]$, the OSC is not satisfied. This is the case, for instance, when
    $J$ contains as a strict subset either $\{(0,2),(0,3),(1,3),(0,4),(1,4)\}$ or $\{(1,2),(0,3),(1,3),(0,4),(1,4)\}$.
\end{ex}

We conclude this section by considering the non-autonomous setting in the case both digits $(0,2)$ and $(1,2)$ are omitted entirely. For $\mathbb{J} = (J_k)_{k\in \mathbb{N}}$ a sequence of finite subsets $J_k \subseteq \{0,1\}\times \mathbb{N}_{\geq 3}$, let $F_{\! \mathds J}$ denote the limit set of the NSIFS $(\{\phi_{s,d} : (s,d)\in J_k\})_{k\in \mathbb{N}}$ acting on $[0,\tfrac12]$. The set $F_{\! \mathds J}$ coincides with a generalised type of restricted digit set
    \begin{align*}
        F_{\! \mathds J} = \{ x \in [0,1] :  x\ \text{has a signed Lüroth expansion with digits}\ (s_k,d_k)\ \text{in} \; J_k\ \text{for each}\ k \in \mathbb{N}\}.
    \end{align*}
Such sets are of particular interest in relation to various questions on the growth rate of the digits $d_k$, as studied for Lüroth expansions in for instance \cite{JR12,CWW13,AG21}. We obtain the following result.

\begin{prop} \label{theorem:Nonautonomous1D}
    If the sequence $\mathds J = (J_k)_{k\in \mathbb{N}}$ with $J_k \subseteq \{0,1\}\times \mathbb{N}_{\geq 3}$ is of \emph{sub-exponential growth}, that is, each set $J_k$ is finite and $\lim_{k \to \infty} \frac1k \log \#J_k = 0$, then
        \begin{align*}
            \dim_{\mathcal{H}} (F_{\!\mathds J}) = \inf\left\{ r\in (0,1] : \liminf_{n \to \infty} \frac1n \sum_{k=1}^n \log\left(\sum_{(s,d)\in J_k} \left(\frac1{d(d-1)}\right)^r \right) <  0\right\}.
        \end{align*}
\end{prop}

\begin{proof}
    Since each $\phi_{s,d}$ is a similarity, the IFS $\{\phi_{s,d} : (s,d)\in J_k\}$ is conformal. Moreover, $\lvert\phi'_{s,d}(x)\rvert = \frac{1}{d(d-1)} \leq \frac{1}{6} < 1$ for each $x\in [0,\tfrac12]$ and hence $(\{\phi_{s,d} : (s,d)\in J_k\})_{k\in \mathbb{N}}$ is uniformly contracting. Therefore, by \cite[Theorem~1.1]{rempe2016non}, the Hausdorff dimension of $F_{\! \mathds J}$ equals $\inf\{r > 0 : \underline{P}(r) < 0\}$. Here $\underline P$ is the \emph{lower pressure function} defined by
        \begin{align*}
            \underline P(r) = \liminf_{m \to \infty}  \frac1m \log \sum_{(s_1,d_1)\in J_1,\ldots,(s_m,d_m)\in J_m} \lVert (\phi_{s_1,d_1} \circ \cdots \circ \phi_{s_m,d_m})'\rVert_\infty^r,
        \end{align*}
    and $\lVert \cdot \rVert_{\infty}$ denotes the supremum norm. For $(s_1,d_1)\in J_1$, $\ldots$, $(s_m,d_m)\in J_m$ it holds that 
        \begin{align*}
            \lVert (\phi_{s_1,d_1} \circ \cdots \circ \phi_{s_m,d_m})' \rVert_\infty = \prod_{k=1}^m \frac{1}{d_k(d_k-1)}
        \end{align*}        
    and so the lower pressure function becomes
        \begin{align*}
            \underline P(r)
            &= \liminf_{m \to \infty}  \frac1m \log \sum_{(s_1,d_1)\in J_1,\ldots,(s_m,d_m)\in J_m} \prod_{k=1}^m \left(\frac{1}{d_k(d_k-1)}\right)^r \\
            &= \liminf_{m \to \infty}  \frac1m \log \left( \prod_{k=1}^m \sum_{(s,d)\in J_k} \left(\frac{1}{d(d-1)}\right)^r \right)\\
            &= \liminf_{m \to \infty}  \frac1m \sum_{k=1}^m \log \left( \sum_{(s,d)\in J_k} \left(\frac{1}{d(d-1)}\right)^r \right). \qedhere
        \end{align*}
\end{proof}

For each subset $J\subseteq \{0,1\}\times \mathbb{N}_{\geq 2}$ and each sequence $\mathbf s = (s_k)_{k\in \mathbb{N}} \in \{0,1\}^{\mathbb{N}}$ we define the subset
$F_{\!J,\mathbf s}$ of $F_{\!J}$ containing the numbers $x\in [0,1]$ for which there exists a sequence $(d_k)_{k\in \mathbb{N}} \in \mathbb{N}_{\geq 2}^{\mathbb{N}}$ such that $((s_k,d_k))_{k\in \mathbb{N}}$ lies in $J^{\mathbb{N}}$ and gives a signed Lüroth expansion for $x$. 

\begin{prop} \label{cor:fibresforsequences}
    Let $I_{0}$ and $I_{1}$ be finite subsets of $\mathbb{N}_{\geq 2}$, let $p\in (0,1)$ and let $\mu_p$ denote the $p$-Bernoulli measure on $\{0,1\}^{\mathbb{N}}$ with $\mu_{p}(\{ \mathbf{s} = (s_{1},s_{2},\dots) \in \{0,1\}^{\mathbb{N}} : s_{1}=0 \}) = p$. If $J = (\{0\}\times I_0)\cup(\{1\}\times I_1)$, then for $\mu_p$-almost every sequence $\mathbf s\in \{0,1\}^{\mathbb{N}}$ it holds that $\dim_{\mathcal H}(F_{J,\mathbf s}) = t$, where $t \in \mathbb{R}_{> 0}$ is the unique solution to 
        \begin{align}\label{eq:prop_Bernoulli}
            \left(\sum_{d_0\in I_0} \left(\frac{1}{d_0(d_0-1)}\right)^t\right)^p \left(\sum_{d_1\in I_1} \left(\frac{1}{d_1(d_1-1)}\right)^t\right)^{1-p} =  1.
        \end{align}
\end{prop}

\begin{proof}
    Let $\mathbf s = (s_k)_{k\in \mathbb{N}} \in \{0,1\}^{\mathbb{N}}$ and set $\mathds J = (J_k)_{k\in \mathbb{N}}$ where $J_k = \{s_k\}\times I_{s_k}$ for $k\in \mathbb{N}$.
    By construction the set $F_{\!J,\mathbf s}$ coincides with the set $F_{\!\mathds J}$. Following similar arguments to \Cref{theorem:Nonautonomous1D}, we have
        \begin{align*}
            \dim_{\mathcal{H}} (F_{\!J,\mathbf s}) = \inf\left\{ r\in (0,1] : \liminf_{n \to \infty} \frac1n \sum_{k=1}^n \log\left(\sum_{d\in I_{s_k}} \left(\frac1{d(d-1)}\right)^r \right) <  0\right\}.
        \end{align*}
    For $n\in \mathbb{N}$ define $\tau_0(\mathbf s,n) = \#\{1\leq k \leq n : s_k = 0\}$ and observe
        \begin{align*}
            \frac1n \sum_{k=1}^n \log\left(\sum_{d\in I_{s_k}} \left(\frac1{d(d-1)}\right)^r\right) &=  \frac{\tau_0(\mathbf s,n)}n \log\left(\sum_{d_0\in I_{0}} \left(\frac1{d_0(d_0-1)}\right)^r\right) + \frac{n - \tau_0(\mathbf s,n)}n \log\left(\sum_{d_1\in I_{1}} \left(\frac1{d_1(d_1-1)}\right)^r\right).
        \end{align*}
    Applying the Birkhoff Ergodic Theorem, where the dynamics is driven by the left-shift map on $\{0,1\}^\mathbb{N}$ and where we take the indicator function on the set $\{ \mathbf{s} = (s_{1},s_{2},\dots) \in \{0,1\}^{\mathbb{N}} : s_{1}=0 \}$ for the observable, we obtain $\lim_{n \to \infty} \frac{\tau_0(\mathbf s,n)}n = p$ for $\mu_p$-almost every $\mathbf{s}\in \{0,1\}^{\mathbb{N}}$. Hence, for such $\mathbf s$,
        \begin{align*}
            \dim_{\mathcal{H}} (F_{\!J,\mathbf s})
            &= \inf\left\{ r\in (0,1] :  p \log\left(\sum_{d_0\in I_{0}} \left(\frac1{d_0(d_0-1)}\right)^r\right) + (1-p) \log\left(\sum_{d_1\in I_{1}} \left(\frac1{d_1(d_1-1)}\right)^r\right) <  0\right\}\\
            &= \inf\left\{ r\in (0,1] :  \left(\sum_{d_0\in I_0} \left(\frac{1}{d_0(d_0-1)}\right)^t\right)^p \left(\sum_{d_1\in I_1} \left(\frac{1}{d_1(d_1-1)}\right)^t\right)^{1-p} <  1\right\}.\qedhere
        \end{align*}
\end{proof}

\subsection{Restricted digit sets with infinite alphabets for signed Lüroth expansions}

We first consider the case when $J\subseteq\{0,1\}\times \mathbb N_{\geq 3}$ and then turn to the case when exactly one of $(0,2)$ or $(1,2)$ lies in $J$.

\begin{theorem} \label{theorem:1DinfiniteJ}
    If $J$ is a countably infinite subset of $\{0,1\} \times \mathbb{N}_{\geq 3}$, then
        \begin{align*}
            \dim_{\mathcal{H}} (F_{\!J})
            &= \inf \left\{ r> 0 : \sum_{(s,d)\in J} \left(\frac{1}{d(d-1)}\right)^{r} \leq 1 \right\}
            \;\; \text{and} \\
            \dim_P(F_{\!J}) &= \overline{\dim}_B(F_{\!J}) = \max\left\{\dim_{\mathcal{H}} (F_{\!J}), \overline{\dim}_B\left(\left\{\frac{1}{d-s} : (s,d)\in J\right\}\right)\right\}.
        \end{align*}
\end{theorem}

\begin{proof}
    For any $(s,d)\in J$ we have $\phi_{s,d}([0,1]) = [\tfrac1d,\tfrac1{d-1}] \subseteq [0,\tfrac12]$ and so the limit set of the IFS $\{\phi_{s,d} : (s,d)\in J\}$ on $[0,1]$ coincides with that of the restricted IFS $\{\phi_{s,d}\vert_{[0,1/2]} : (s,d)\in J\}$ on $[0,\tfrac12]$ and satisfies the OSC with feasible open set $(0,\tfrac12)$. One readily checks that the restricted IFS satisfies the conditions of \cite[Corollary~3.17]{mauldin1996dimensions} and thus that $\dim_{\mathcal{H}} (F_{\!J}) = \inf\{ r> 0 : P(r) \leq 0\}$. Here $P\colon \mathbb{R}_{> 0} \to \mathbb{R} \cup \{\infty\}$ is the \emph{pressure function} defined, for $r > 0$, by
        \begin{align*}
            P(r) = \lim_{m \to \infty}  \frac1m \log \left( \sum_{(s_1,d_1),\ldots,(s_m,d_m)\in J} \lVert (\phi_{s_1,d_1} \circ \cdots \circ \phi_{s_m,d_m})'\rVert_{\infty}^r \right).
        \end{align*}
    Since $\lVert(\phi_{s_1,d_1} \circ \cdots \circ \phi_{s_m,d_m})'\rVert_{\infty} = \prod_{k=1}^m \frac{1}{d_k(d_k-1)}$ for any $(s_1,d_1),\ldots,(s_m,d_m)\in J$, this becomes
        \begin{align}\label{eq:original_pressure}
            \begin{aligned}
                P(r)
                &= \lim_{m \to \infty}  \frac{1}{m} \log \sum_{(s_1,d_1),\ldots,(s_m,d_m)\in J} \prod_{k=1}^m \left(\frac{1}{d_k(d_k-1)}\right)^r\\
                &= \lim_{m \to \infty}  \frac1m \log \left( \sum_{(s,d)\in J} \left(\frac{1}{d(d-1)}\right)^r \right)^m 
                = \log \sum_{(s,d)\in J} \left(\frac{1}{d(d-1)}\right)^r.
            \end{aligned}
        \end{align}
    Therefore, $P(r) \leq 0$ if and only if $\sum_{(s,d)\in J} (\frac{1}{d(d-1)})^{r} \leq 1$, yielding the result for the Hausdorff dimension.
    
    The equality of the packing and upper box-counting dimensions follows from \cite[Theorem~3.1]{mauldin1996dimensions} and the formula for the upper box-counting dimension is a consequence of \cite[Theorem~2.11]{MU02}. 
    \end{proof}

\begin{remark}\label{rem:intermediate}
    \Cref{theorem:1DinfiniteJ} in tandem with \cite[Theorem~3.5 and Corollary~3.6]{BF2023} yields that if $J$ is a countably infinite subset of $\{0,1\} \times \mathbb{N}_{\geq 3}$, then for $\theta \in [0,1]$,
        \begin{align*}
            \max\left\{\dim_{\mathcal{H}} (F_{\!J}),\
            \underline{\dim}_{\theta}\left(\left\{\frac{1}{d-s} : (s,d)\in J\right\}\right)\right\}
            &\leq \underline{\dim}_\theta(F_{\!J})\\[-0.5em]
            &\leq \overline{\dim}_\theta(F_{\!J})
            = \max\left\{\dim_{\mathcal{H}} (F_{\!J}),\ \overline{\dim}_\theta\left(\left\{\frac{1}{d-s} : (s,d)\in J\right\}\right)\right\},
        \end{align*} 
    where $\underline{\dim}_{\theta}$ and $\overline{\dim}_{\theta}$ respectively denote the lower and upper intermediate dimensions. Moreover, applying the same analysis as in \cite[Theorem~4.3(2)]{BF2023} to the sets $F_{\!J}$ yields that the maps $\theta \mapsto \underline{\dim}_\theta(F_{\!J})$ and $\theta \mapsto \overline{\dim}_\theta(F_{\!J})$ are continuous at $\theta = 0$. For a formal definition of the lower and upper intermediate dimensions we refer the reader to \cite{MR4140764}, where they were first introduced, and where it was noted that $\underline{\dim}_{\theta}(F)= \overline{\dim}_{\theta}(F) = \dim_{\mathcal{H}}(F)$ when $\theta = 0$, and $\underline{\dim}_{\theta}(F)= \underline{\dim}_{B}(F)$ and $\overline{\dim}_{\theta}(F)= \overline{\dim}_{B}(F)$ when $\theta = 1$, for any bounded set $F \subseteq \mathbb{R}^{D}$.
\end{remark}

We treat some examples where the limit set $F_{\!J}$ has equal Hausdorff and box-counting dimensions as well as examples where the Hausdorff dimension is strictly smaller than the box-counting dimension.

\begin{ex}
    Suppose $J = (\{0,1\}\times \mathbb{N}_{\geq 3})\setminus S$ for some finite set $S$. For any $r\leq \tfrac{1}{2}$ we have
        \begin{align*}
            \sum_{(s,d)\in J} \left (\frac{1}{d(d-1)} \right)^r \geq \sum_{(s,d)\in J} \left (\frac{1}{d^2} \right)^r \geq \sum_{(s,d)\in J} \frac{1}{d} = 2\sum_{d=3}^\infty \frac{1}{d} - \sum_{(s,d)\in S} \frac{1}{d} = \infty.
        \end{align*}
    It therefore follows from \Cref{theorem:1DinfiniteJ} that $\dim_{\mathcal H}(F_{\!J}) = \inf \{ r : \sum_{(s,d)\in J} (\tfrac{1}{d(d-1)})^{r} \leq 1 \} \geq \tfrac12$. Further, we have that $\tfrac12 = \overline{\dim}_B(\{\frac{1}{n} : n\in \mathbb{N}\}) \geq \overline{\dim}_B(\{\frac{1}{d-s} : (s,d)\in J\})$, so by \Cref{theorem:1DinfiniteJ} we have $\dim_{\mathcal H}(F_{\!J}) = \dim_B(F_{\!J})$.
\end{ex}

\begin{ex}
    If $J\subseteq \{0,1\}\times \mathbb{N}_{\geq 3}$ is such that $\{\frac{1}{d-s} : (s,d)\in J\} = \{\frac{1}{n^k} : n\in \mathbb{N}_{\geq 2}\}$ for some $k\in \mathbb{N}$, then $\dim_B(\{\frac{1}{d-s} : (s,d)\in J\}) = \frac1{k+1}$, see for instance \cite[Example~3.1]{MR3236784}. \Cref{theorem:1DinfiniteJ} in tandem  with \Cref{rem:intermediate} then implies $\dim_B(F_{\!J})$ exists and equals $\max\{\dim_{\mathcal H}(F_{\!J}), \frac1{1+k}\}$.

    For instance, if $J = \{(0,n^k) : n\in \mathbb{N}_{\geq 2}\}$ for some $k \in \{ 2, 3, 4, 5, 6\}$, then $\frac{2k}{k+1} \leq \frac{2\cdot 6}{6+1} = \frac{12}7$, and so for each $r\leq \frac{1}{k+1}$,
        \begin{align*}
            \sum_{(s,d)\in J} \left( \frac{1}{d(d-1)} \right)^r 
            \geq \sum_{n=2}^\infty \left(\frac{1}{n^k(n^k-1)}\right)^{\frac{1}{k+1}} \geq \sum_{n=2}^\infty \frac{1}{n^{\frac{2k}{k+1}}}
            \geq \sum_{n=2}^\infty \frac{1}{n^{\frac{12}{7}}} 
            =\zeta\left( \frac{12}{7}\right) -1> 1,
        \end{align*}
    with $\zeta$ denoting the Riemann $\zeta$-function.
    It follows that $\dim_{\mathcal H}(F_{\!J}) \geq \frac1{k+1}$ and hence $\dim_{\mathcal H}(F_{\!J}) = \dim_B(F_{\!J})$.

    If instead we take $J = \{(1,n^k+1) : n\in \mathbb{N}_{\geq 2}\}$ for some integer $k \geq 7$, then for any $r>\frac{1}{2k}$ we have
        \begin{align*}
            \sum_{(s,d)\in J} \left( \frac{1}{d(d-1)} \right)^r = \sum_{n=2}^\infty \left(\frac{1}{(n^k+1)n^k}\right)^{r} \leq \sum_{n=2}^\infty \frac{1}{n^{2kr}}
            =\zeta\left( 2kr\right) -1.
        \end{align*}
    In particular, for any $r\geq \frac{1}{k+1} - \frac{1}{100k}$ we have that $\zeta\left( 2kr\right) -1 \leq \zeta\left(\frac{2k}{k+1} - \frac1{50}\right) - 1 \leq \zeta\left(\frac{7}4 - \frac{1}{50}\right)-1 < 1$ and hence $\dim_{\mathcal H}(F_{\!J}) \leq \frac{1}{k+1} - \frac{1}{100k} < \frac{1}{k+1} = \dim_{B}(F_{\!J})$.
\end{ex}

\begin{remark}\label{rem:infiniteOSC}
    Whenever $\{\phi_{s,d} : (s,d)\in J\}$, with $J$ countably infinite, satisfies the OSC with a feasible open set consisting of finitely many open intervals, one could attempt to show that the result of \Cref{theorem:1DinfiniteJ} holds by representing the system as an infinitely generated conformal graph-directed system in the sense of \cite{MR2003772} and by applying the results therein. However, if $J$ is infinite and contains either $(0,2)$ or $(1,2)$, but not both, the OSC can only be satisfied with feasible open set $(0,1)$ or with a  feasible open set consisting of an infinite union of disjoint open intervals. The former is the case when for each $d \in \mathbb{N}_{\geq 2}$ the alphabet $J$ contains at most one of the digits $(0,d)$ and $(1,d)$, while the latter is the case whenever there is at least one digit $d \in \mathbb{N}_{\geq 2}$ for which $(0,d),(1,d)\in J$. 
    
    To see this, note that if $\{\phi_{s,d} : (s,d) \in J\}$ satisfies the OSC with feasible open set $U$, then for each $(s,d) \in J$, the open set $U$ must have a non-empty intersection with $\phi_{s,d}([0,1]) = [\tfrac1d,\tfrac{1}{d-1}]$. Hence, assuming $J$ is infinite and $U$ is a finite union of open intervals, then one of these intervals must be of the form $(0,\varepsilon)$ for some $\varepsilon \in (0,1]$. If $(1,2)\in J$ this means $U$ must also contain the interval $\phi_{1,2}((0,\varepsilon)) = (1-\tfrac12\varepsilon, 1)$. If instead $(0,2)\in J$ then $U$ must contain $\bigcup_{k\in \mathbb{N}} \phi_{0,2}^k((0,\varepsilon)) = \bigcup_{k\in \mathbb{N}} (1-\tfrac1{2^k},\tfrac1{2^k}\varepsilon + 1 - \tfrac1{2^k})$. In both cases the assumption that $U$ is a finite union of open intervals yields that $U$ must contain intervals $(0,\varepsilon)$ and $(\delta,1)$ for some $\varepsilon \in (0,1]$ and $\delta \in [0,1)$. However, if $d\in \mathbb{N}_{\geq 2}$ is such that $(0,d),(1,d)\in J$ then $\phi_{0,d}((0,\varepsilon)) \cap \phi_{1,d}((\delta,1)) = (\tfrac1d,\tfrac1{d(d-1)}\varepsilon + \tfrac1d) \cap (\tfrac1d, \tfrac{1}{d-1}-\tfrac1{d(d-1)}\delta) \neq \emptyset$, meaning $U$ is not a feasible open set for the IFS $\{\phi_{s,d} : (s,d)\in J\}$ to satisfy the OSC.
\end{remark}

As a corollary to \Cref{theorem:1DinfiniteJ} and \cite[Corollary~6.8]{Urbanski_dimension_spectrum} we obtain in the following result that the dimension spectra of the IFSs $\{\phi_{0,d} : d \in \mathbb{N}_{\geq 2} \}$ and $\{\phi_{1,d} : d \in \mathbb{N}_{\geq 2} \}$ are full, which we utilise in the proof of \Cref{thm:dimension_spectrum}.

\begin{cor}\label{cor:DS_full_1D}
    For $s \in \{0,1\}$ we have $\{ \dim_{\mathcal{H}}(F_{\!J}) : J \subseteq \{s\} \times \mathbb{N}_{\geq 2} \} = [0,1]$.
\end{cor}

\begin{proof}
    By \cite[Corollary~6.8]{Urbanski_dimension_spectrum}, it is sufficient to verify that (i) $\sum_{d \in \mathbb{N}_{\geq k+1}} \left( \frac{1}{d(d-1)} \right)^{t} \geq \left( \frac{1}{k(k-1)} \right)^{t}$, for all $k \in \mathbb{N}_{\geq 3}$ and $t \in (0,1)$ and (ii) $\inf\{t> 0 : P(t)\leq 0\}=1$ with $P$ as in \eqref{eq:original_pressure} and $J=\{s\}\times\mathbb N_{\geq 2}$. Part (i) follows from the fact that, for $k \in \mathbb{N}_{\geq 3}$, the map $t \mapsto \sum_{d \in \mathbb{N}_{\geq k+1}} \left( \tfrac{k(k-1)}{d(d-1)} \right)^{t}$, where it is well defined on $(0,1]$, is monotonically decreasing, and that $\sum_{d \in \mathbb{N}_{\geq k+1}} \tfrac{k(k-1)}{d(d-1)} = k-1 \geq 1$.
    Part (ii) follows from $t\mapsto \sum_{d \in \mathbb{N}_{\geq 2}} \left( \frac{1}{d(d-1)} \right)^{t}$ being monotonically decreasing and $\sum_{d \in \mathbb{N}_{\geq 2}} \frac{1}{d(d-1)} =1$.
\end{proof}

\section{\textbf{Hausdorff and affinity dimensions of \texorpdfstring{$\mathcal{F}^p_{\!J}$}{Calligraphic F underscore J}}}\label{s:theorem2}

We now consider the self-affine sets $\mathcal{F}^p_{\!J}$ generated by the iterated function systems $\{A^p_{s,d} : (s,d)\in J\}$ on $[0,1]^2$ for countable alphabets $J\subseteq  \{0,1\} \times \mathbb{N}_{\geq 2}$. Notably, each IFS $\{A^p_{s,d} : (s,d)\in J\}$ satisfies the ROSC and the linear part of each affine map $A^p_{s,d}$ is the diagonal matrix
    \begin{align*}
        L^p_{s,d} = \begin{pmatrix}
        p^{1-s}(1-p)^s & 0 \\
        0 & (-1)^s\frac{1}{d(d-1)}
        \end{pmatrix}.
    \end{align*}
Applying \Cref{theorem:MorrisInfinite} to this setting yields the following expression for the affinity dimension of $\{L^p_{s,d} : (s,d)\in J\}$.

\begin{prop}\label{cor:affinityLuroth}
    Let $J\subseteq \{0,1\}\times \mathbb{N}_{\geq 2}$ be a finite or countably infinite alphabet satisfying $\pi_1(J) = \{0,1\}$ and let $p\in(0,1)$ arbitrary. The affinity dimension $d(L^p_{s,d} \ \vert \ (s,d)\in J)$ of $\{L^p_{s,d} : (s,d)\in J\}$ lies in $[1,2]$ and equals
        \begin{align}\label{eq:affinity_projection}
            \inf\left\{ r \in (1,2] : \max\left\{\sum_{(s,d)\in J}p^{1-s}(1-p)^s\left(\frac{1}{d(d-1)}\right)^{r-1},\ \sum_{(s,d)\in J} (p^{1-s}(1-p)^s)^{r-1} \left(\frac{1}{d(d-1)}\right) \right\} \leq 1 \right\}.
        \end{align}
    If (a) $\sum_{(s,d)\in J} \frac{1}{d(d-1)} \leq 1$ or (b) $J = \{0,1\} \times I$ for some $I \subseteq \mathbb{N}_{\geq2}$, or (c) $p = \frac{1}{2}$, then this formula simplifies to
        \begin{align*}
            d(L^p_{s,d} \ \vert \ (s,d)\in J) = \inf\left\{r \in (1,2] : \sum_{(s,d)\in J}p^{1-s}(1-p)^s\left(\frac{1}{d(d-1)}\right)^{r-1} \leq 1\right\}.
        \end{align*}
\end{prop}

\begin{proof}
    Since by assumption $\pi_1(J) = \{0,1\}$, we have for each $0 < r < 1$ that 
        \begin{align*}
            \max\left\{ \sum_{(s,d)\in J} (p^{1-s}(1-p)^s)^r, \sum_{(s,d)\in J} \left(\frac1{d(d-1)}\right)^r  \right\}
            \geq \sum_{(s,d)\in J} (p^{1-s}(1-p)^s)^r > \sum_{(s,d)\in J} p^{1-s}(1-p)^s
            \geq 1.
        \end{align*}
    This implies that $\inf \{ r > 0 : \max \{ \sum_{(s,d)\in J} (p^{1-s}(1-p)^s)^r, \sum_{(s,d)\in J} (\tfrac1{d(d-1)})^r \} \leq 1 \} \geq 1$, and so by \Cref{theorem:MorrisInfinite} the affinity dimension of $\{L^p_{s,d} : (s,d)\in J\}$ is at least $1$. Further, for $r\geq 2$ we have
        \begin{align*}
            \sum_{(s,d)\in J} \left(\frac{p^{1-s}(1-p)^s}{d(d-1)}\right)^{r/2} \leq \sum_{(s,d)\in J} \frac{p^{1-s}(1-p)^s}{d(d-1)} \leq \sum_{d\in \pi_2(J)} \frac{p+1-p}{d(d-1)} \leq \sum_{d\in \mathbb{N}_{\geq 2}} \frac{1}{d(d-1)} = 1.
        \end{align*}
    Thus, $\inf \{ r> 2 : \sum_{(s,d)\in J} (\frac{p^{1-s}(1-p)^s}{d(d-1)})^{r/2} \leq 1\} = 2$, and so \Cref{theorem:MorrisInfinite} implies the affinity dimension is at most $2$ and equals the quantity given in\eqref{eq:affinity_projection}.
    
    For the simplification, note that whenever $p=1/2$, we have that $p^{1-s}(1-p)^s = \frac12 \geq \frac{1}{d(d-1)}$ for all $(s,d)\in J$. Hence
        \begin{align*}
            \sum_{(s,d)\in J} p^{1-s}(1-p)^s \left(\frac1{d(d-1)}\right)^{r-1} \geq \sum_{(s,d)\in J} \left(p^{1-s}(1-p)^s\right)^{r-1}\frac1{d(d-1)}, 
        \end{align*}
    for all $r \in [1,2]$, yielding (c). When $p\in (0,1)$ and $\sum_{(s,d)\in J} \frac{1}{d(d-1)} \leq 1$, observe that for $r \in [1, 2]$, 
        \begin{align*}
            \sum_{(s,d)\in J} (p^{1-s}(1-p)^s)^{r-1} \frac{1}{d(d-1)} \leq \sum_{(s,d)\in J} \frac{1}{d(d-1)} \leq 1,
        \end{align*}
    yielding (a). For (b), by symmetry, we may assume that $p \in (0, 1/2]$. For $n\in \mathbb N_{\geq 2}$ and $r\in [1,2]$, we show that
         \begin{align}\label{eq:eq:dimension_spectra_cor_1}
            \frac{1}{n^{r-1}} \geq (p^{r-1}+(1-p)^{r-1}) \frac{1}{n}.
        \end{align}     
    This inequality holds if and only if $n \geq q^{r-1} + (n-q)^{r-1}$, where $q = np$. Let $g_{r}\colon [0,\frac{n}{2}] \to \mathbb{R}$ be defined by $g_{r}(q) = q^{r-1} + (n-q)^{r-1}$, and note, by the first derivative test, that $g_{r}$ is maximised at $q = \tfrac{n}{2}$. This implies for all $q\in (0,\tfrac{n}2]$, and hence $p\in (0,\tfrac12]$, that $q^{r-1} + (n-q)^{r-1} = g_r(q) \leq g_r(\tfrac{n}2) = 2(\tfrac{n}2)^{r-1} \leq 2\tfrac{n}2 = n$. By the assumption that $(0,d) \in J$ implies $(1,d) \in J$ and vice versa, and using \eqref{eq:eq:dimension_spectra_cor_1} with $n=d(d-1)$, we conclude
        \[
            \sum_{(s,d)\in J}p^{1-s}(1-p)^s\left(\frac{1}{d(d-1)}\right)^{r-1}
            \!=\! \sum_{d\in I}\left(\frac{1}{d(d-1)}\right)^{r-1}
            \!\geq\! \sum_{d\in I} \frac{p^{r-1} + (1-p)^{r-1}}{d(d-1)}
            \!=\! \sum_{(s,d)\in J} (p^{1-s}(1-p)^s)^{r-1}\frac{1}{d(d-1)}.\qedhere
        \]
\end{proof}

Defining the maps $f^p_0(w) = pw$ and $f^p_1(w) = (1-p)w + p$ for $w\in [0,1]$ we note that $A^p_{s,d}(w,x) = (f^p_s(w), \phi_{s,d}(x))$ for each $(w,x)\in [0,1]^2$. As such, the horizontal projection $\pi_1(\mathcal F^p_J)$ is exactly the self-similar set of the IFS $\{f_s^p : s\in \pi_1(J)\}$. In particular, whenever $\pi_1(J) = \{0,1\}$, we have $\pi_1(\mathcal F^p_J) = [0,1]$ and hence $\dim_B(\pi_1(\mathcal F^p_J)) = 1$. In the same way, the vertical projection $\pi_2(\mathcal F^p_J)$ equals $F_J$, the self-similar set of $\{\phi_{s,d} : (s,d)\in J\}$ discussed in \Cref{section:1D}. Under suitable conditions the dimension $\dim_B(\pi_2(\mathcal F^p_J))$ is given by \eqref{eq:theorem:1DfiniteJ} when $J$ is finite. With this in mind we obtain the following result whenever $(0,2),(1,2)\notin J$.

\begin{lemma}\label{lemma:modifiedLuroth}
    If $J \subseteq \{0,1\}\times \mathbb{N}_{\geq 2}$ is a finite alphabet such that either $\pi_{2}(J) \subseteq \mathbb{N}_{\geq 3} $ and $\pi_1(J) = \{0,1\}$, or $\{ (0,2), (1,2) \} \subseteq J$, then $\dim_B (\mathcal F_J^p) = \dim_P (\mathcal F_J^p) = d(L^p_{s,d} \ \vert \ (s,d)\in J)$ holds for all $p \in (0,1)$.
\end{lemma}

\begin{proof}
    If $\{ (0,2), (1,2) \} \subseteq J$, then 
    by construction and by \eqref{eq:dimension_1D_0_2_and_1_2}, $\dim_B(\pi_1(\mathcal F^p_J)) = \dim_B(\pi_2(\mathcal F^p_J)) = 1$, and thus the result is an application of \Cref{theorem:FraserSimplified} in combination with the second part of \Cref{cor:affinityLuroth}. Therefore, let us consider the case when $\pi_{2}(J) \subseteq \mathbb{N}_{\geq 3} $ and $\pi_1(J) = \{0,1\}$. Set $r_1 = \dim_B(\pi_1(\mathcal F^p_J))$ and $r_2 = \dim_B(\pi_2(\mathcal F^p_J))$. By the assumption $\pi_1(J) = \{0,1\}$ we have $r_1 = 1$, and by \eqref{eq:theorem:1DfiniteJ} together with the assumption $\pi_2(J)\subseteq \mathbb N_{\geq 3}$, which implies the OSC, $r_2$ uniquely solves $\sum_{(s,d)\in J} (\frac1{d(d-1)})^{r_2} = 1$. By \Cref{theorem:FraserSimplified} we have $\dim_B \mathcal F_J^p = \dim_P \mathcal F_J^p = r'$, where $r'$ solves
        \begin{align*} 
            \max\Bigg{\{}\underbrace{\sum_{(s,d)\in J} p^{1-s}(1-p)^s\left(\frac1{d(d-1)}\right)^{r' - 1}}_{= v_1(r')},\ \underbrace{\sum_{(s,d)\in J} (p^{1-s}(1-p)^s)^{r'-r_2}\left(\frac1{d(d-1)}\right)^{r_2} }_{=v_2(r')}\Bigg{\}} = 1.
        \end{align*}
    Next, we show that $r'= d(L_{s,d}^p\mid (s,d)\in J)$. For this, note that $v_1(r), v_2(r)$ are decreasing in $r$ and that 
        \begin{align*}
            v_2(r_2) = \sum_{(s,d)\in J} \left(\frac1{d(d-1)}\right)^{r_2} = 1.
        \end{align*}
    In particular, since $r_2\leq 1$, we have $v_2(r) \leq 1$ for all $r\geq 1$. Observe that the assumption $\pi_2(J)\subseteq \mathbb N_{\geq 3}$ implies $\sum_{(s,d)\in J} \frac{1}{d(d-1)} \leq \sum_{(s,d)\in \{0,1\}\times \mathbb{N}_{\geq 3}} \frac{1}{d(d-1)} = 1$. Thus, by the second part of \Cref{cor:affinityLuroth}, the number $r$ solving $v_1(r) = 1$ satisfies $r\geq 1$. Since both $v_1$ and $v_2$ are decreasing in $r$, we deduce that $v_1(r')\geq v_2(r')$ giving $v_1(r')=1$. Therefore, it follows from the second part of \Cref{cor:affinityLuroth} that $r' = d(L^p_{s,d} \ \vert \ (s,d)\in J)$.
\end{proof}

\begin{proof}[Proof of \Cref{theorem:2DF}]
    We begin by showing \eqref{eq:dimesion_of_F_J_p} and divide the argument into two cases, when $J$ is finite, and when $J$ is countably infinite.  To this end, let us assume that $J$ is finite and let $E = \left\{ \mathbf s \in \{0,1\}^\mathbb{N} :  \lim_{n \to \infty} \frac{\tau_0(\mathbf s,n)}{n}=p \right\}$, where $\tau_0$ is as in the proof of \Cref{cor:fibresforsequences}. We have seen in the proof of \Cref{cor:fibresforsequences} that for all $\mathbf s \in E$ the Hausdorff dimension of $F_{\!J,\mathbf s}$ does not depend on $\mathbf s \in E$ and is given by the unique $t$ solving \eqref{eq:prop_Bernoulli}. As in \Cref{s:sss}, let $\pi: \{0,1\}^\mathbb{N} \to [0,1]$ denote the projection map given by $\pi (\mathbf s) = \lim_{n \to \infty} f_{s_1}^p \circ \cdots \circ f_{s_n}^p(0)$. Observe that $\pi\vert_{E}: E \to \pi(E)$ is a bijection, and that $\mu_p(E)=1$ by the Birkhoff Ergodic Theorem. Let $\xi_p : [0,1]\rightarrow [0,1]$ be defined by
        \begin{align*}
            \xi_{\!p}(w)
                = \begin{cases}
                    \tfrac{w}{p} & \text{if} \; w\in [0,p],\\
                    \tfrac{w}{1-p} - \tfrac{p}{1-p} & \text{if} \; w\in (p,1],
                \end{cases}
        \end{align*}
    (so $f_0^p$ and $f_1^p$ are the local inverses of $\xi_p$). Let $\lambda$ denote the Lebesgue measure on the Borel $\sigma$-algebra of $[0,1]$ and denote the left-shift by $\sigma: \{0,1\}^{\mathbb{N}} \rightarrow \{0,1\}^{\mathbb{N}}$. The dynamical systems $([0,1],\lambda,\xi_p)$ and $(\{0,1\}^{\mathbb{N}}, \mu_p,\sigma)$ are measure theoretically isomorphic through the map $\pi: \{0,1\}^{\mathbb{N}} \rightarrow [0,1]$, see for instance \cite{MR1449135}, and thus $\lambda(\pi(E))=1$.

    For $w\in \pi(E)$, let $(\mathcal F_{\!J}^p)_w = \{x\in [0,1] : (w,x)\in \mathcal F_{\!J}^p\}$ be the vertical fibre of $\mathcal F_{\!J}^p$ based at $w$. Since there is a unique $\mathbf s = (s_k)_{k \in \mathbb{N}} \in \{0,1\}^\mathbb{N}$ with $\pi(\mathbf s)=w$ and each $(w,x)\in \mathcal F_{\!J}^p$ gives a signed Lüroth expansion of $x$ with digit sequence $((s_k,d_k))_{k\in\mathbb{N}} \in J^{\mathbb{N}}$ via $(w,x) = \lim_{k\rightarrow \infty} (A_{s_1,d_1}^p \circ \dots \circ A_{s_{k},d_{k}}^p)((0,0))$, we have $(\mathcal F_{\!J}^p)_w = F_{\!J,\mathbf s}$. Hence, \Cref{cor:fibresforsequences} implies $\dim_{\mathcal H}((\mathcal F_{\!J}^p)_w) = \dim_{\mathcal H}(F_{J,\pi^{-1}(w)}) = t$, for $\lambda$-almost all $w\in [0,1]$, where $t \in \mathbb{R}_{> 0}$ uniquely solves
        \begin{align*}
            \mathfrak{p}(I_0,I_1,t)
            = \left(\sum_{d_0\in I_0} \left(\frac{1}{d_0(d_0-1)}\right)^t\right)^p \left(\sum_{d_1\in I_1} \left(\frac{1}{d_1(d_1-1)}\right)^t\right)^{1-p}
            = 1.
        \end{align*}
    Since $t$ does not depend on $w$, and since this holds for all $w$ in a set of positive Lebesgue measure (and hence of Hausdorff dimension $1$), it is a direct consequence of \cite{Mar54} that $\dim_{\mathcal H}(\mathcal F_{\!J}^p) \geq 1 + t$, yielding \eqref{eq:dimesion_of_F_J_p}.

    Suppose that $J$ is countably infinite. For $n\in \mathbb{N}_{\geq2}$ and $s\in \{0,1\}$ let $I_{n,s} = I_s\cap \{2,\ldots,n\}$, and let $J_n$ denote the set $(\{0\}\times I_{n,0})\cup(\{1\}\times I_{n,1})$. Set $t_n \in \mathbb{R}_{> 0}$ to be the unique solution to $\mathfrak{p}(I_{0,n},I_{1,n},t_n) = 1$ and observe that $(t_n)_{n \in \mathbb{N}_{\geq 2}}$ is a non-decreasing sequence in $[0,1]$. This, in tandem with the fact that $\mathcal F_{\!J}^p \supseteq \mathcal F_{\!J_{n+1}}^p \supseteq \mathcal F_{\!J_{n}}^p$ for all $n \in \mathbb{N}_{\geq 2}$, yields $\dim_{\mathcal H}(\mathcal F_{\!J}^p) \geq \sup_{n\geq 2} \dim_{\mathcal H}(\mathcal F_{\!J_n}^p) \geq 1 + \sup_{n\geq 2} t_n$.

    Letting $t = \sup_{n\geq 2} t_n$, we observe for $n \in \mathbb{N}_{\geq 2}$ that $\mathfrak{p}(I_{0,n},I_{1,n},t) \leq 1$. Taking the limit as $n$ tends to infinity yields \eqref{eq:dimesion_of_F_J_p}. To conclude the proof, we show \eqref{eq:dimesion_of_F_J_p_2}. Since $I_0=I_1=I$, it holds by the second part of \Cref{cor:affinityLuroth} that
        \begin{align*}
            d(L^p_{s,d} \ \vert \ (s,d)\in J)
            &= \inf\left\{r \in (1,2] : \sum_{(s,d)\in J}p^{1-s}(1-p)^s\left(\tfrac{1}{d(d-1)}\right)^{r-1} \leq 1\right\}\\ 
            &= \inf\left\{r \in (1,2] : \sum_{d\in I}\left(\tfrac{1}{d(d-1)}\right)^{r-1} \leq 1\right\}\\
            &= 1 +  \inf\left\{r \in (0,1] : \sum_{d\in I}\left(\tfrac{1}{d(d-1)}\right)^{r} \leq 1\right\} \leq 2.
    \intertext{By \eqref{eq:dimesion_of_F_J_p}, the assumption $I_{0} = I_{1} = I$ also yields}
            \dim_{\mathcal H}(\mathcal F_{\!J}^p) &\geq 1 + \inf\left\{ r\in (0,1] :  \left(\sum_{d_0\in I} \left(\frac{1}{d_0(d_0-1)}\right)^r\right)^p \left(\sum_{d_1\in I} \left(\frac{1}{d_1(d_1-1)}\right)^r\right)^{1-p} \leq  1\right\}\\
            &= 1 + \inf\left\{ r\in (0,1] :  \sum_{d\in I} \left( \frac{1}{d(d-1)}\right)^r \leq  1\right\}
            = d(L^p_{s,d} \ \vert \ (s,d)\in J).
        \end{align*}
    This in tandem with \Cref{prop:lower bound}\eqref{item:prop:lower_bound_1} yields \eqref{eq:dimesion_of_F_J_p_2}.  Moreover, if $I$ is finite, then \Cref{lemma:modifiedLuroth} in combination with \Cref{prop:lower bound}\eqref{item:prop:lower_bound_1} gives that $\dim_{\mathcal H}(\mathcal F_{\!J}^p) = \dim_{P}(\mathcal F_{\!J}^p) = \dim_{B}(\mathcal F_{\!J}^p) = d(L^p_{s,d} \ \vert \ (s,d)\in J)$.
\end{proof}

\begin{proof}[Proof of \Cref{thm:dimension_spectrum}]
    Let $t\in [0,1]$ be chosen arbitrarily. For fixed $I \subseteq \mathbb{N}_{\geq 2}$ and $s \in \{0, 1\}$ the limit set of $\{\phi_{s,d} : d \in I \}$ equals $F_{\!\{s\}\times I}$ whereas that of $\{ A_{s, d}^{p} : d \in I \}$ equals $\{s\}\times F_{\!\{s\}\times I}$, meaning the two have equal Hausdorff dimensions.
    By \Cref{cor:DS_full_1D} we can find a set $I \subseteq \mathbb N_{\ge 2}$ such that $\dim_{\mathcal H} (\mathcal F_{\!\{0\}\times I}^p) = \dim_{\mathcal H} (\mathcal F_{\!\{1\}\times I}^p)=t$. With this at hand, \Cref{theorem:2DF,theorem:1DinfiniteJ} together with \Cref{cor:DS_full_1D} imply $\dim_{\mathcal H} (\mathcal F_{\!\{0,1\}\times I}^p) = 1+t$.
\end{proof}


\begin{thebibliography}{FLMW10}

\bibitem[AGR21]{AG21}
A.~Arroyo and G.~Gonz\'{a}lez~Robert.
\newblock Hausdorff dimension of sets of numbers with large {L}\"{u}roth
  elements.
\newblock {\em {Integers}}, 21:Paper No. A71, 20, 2021.

\bibitem[BF23]{BF2023}
A.~Banaji and J.~Fraser.
\newblock Intermediate dimensions of infinitely generated attractors.
\newblock {\em {Trans. Amer. Math. Soc.}}, 376:2449--2479, 2023.

\bibitem[BI09]{BI09}
L.~Barreira and G.~Iommi.
\newblock Frequency of digits in the {L}\"{u}roth expansion.
\newblock {\em {J. Number Theory}}, 129:1479--1490, 2009.

\bibitem[BBDK96]{BBDK1996}
J.~Barrionuevo, R.~Burton, K.~Dajani, and C.~Kraaikamp.
\newblock Ergodic properties of generalised {L}{\"u}roth series.
\newblock {\em {Acta Arith.}}, LXXIV:311--327, 1996.

\bibitem[Bed84]{PhD_Bedford}
T.~Bedford.
\newblock {\em {Crinkly curves, {M}arkov partitions and dimension}}.
\newblock PhD thesis, University of Warwick, UK, 1984.

\bibitem[BK24]{BK23}
A.~Boonstra and C.~Kalle.
\newblock Constructions of normal numbers with infinitely many digits.
\newblock arxiv:2402.14500v1, 2024.

\bibitem[CWW13]{CWW13}
C.-Y. Cao, B.-W. Wang, and J.~Wu.
\newblock \textls[-10]{The growth speed of digits in infinite iterated function systems.}
\newblock \textls[-10]{{\em {Studia Math.}}, 217:139--158, 2013.}

\bibitem[CLU19]{Urbanski_dimension_spectrum}
V.~Chousionis, D.~Leykekhman, and M.~Urba\'nski.
\newblock The dimension spectrum of conformal graph directed {M}arkov systems.
\newblock {\em {Selecta Math.}}, 25, 2019.

\bibitem[DK09]{DK09}
K.~Dajani and C.~Kalle.
\newblock A natural extension for the greedy {$\beta$}-transformation with
  three arbitrary digits.
\newblock {\em {Acta Math. Hungar.}}, 125:21--45, 2009.

\bibitem[Fal88]{falconer_1988}
K.~Falconer.
\newblock The {H}ausdorff dimension of self-affine fractals.
\newblock {\em {Math. Proc. Cambridge Philos. Soc.}}, 103:339--350, 1988.

\bibitem[Fal97]{MR1449135}
K.~Falconer.
\newblock {\em {{T}echniques in fractal geometry}}.
\newblock John Wiley and Sons, Ltd., Chichester, 1997.

\bibitem[Fal14]{MR3236784}
K.~Falconer.
\newblock {\em {{F}ractal geometry}}.
\newblock John Wiley \& Sons, Ltd., Chichester, third edition, 2014.

\bibitem[FFK20]{MR4140764}
K.~Falconer, J.~Fraser, and T.~Kempton.
\newblock Intermediate dimensions.
\newblock {\em {Math. Z.}}, 296:813--830, 2020.

\bibitem[FLMW10]{FLMW10}
A.~Fan, L.~Liao, J.~Ma, and B.~Wang.
\newblock Dimension of {B}esicovitch-{E}ggleston sets in countable symbolic
  space.
\newblock {\em {Nonlinearity}}, 23:1185--1197, 2010.

\bibitem[FZ23]{FZ23}
Y.~Feng and Q.~Zhou.
\newblock Dimension theory of {L}\"{u}roth digits.
\newblock {\em {Acta Math. Hungar.}}, 170:150--167, 2023.

\bibitem[Fra12]{Fraser_2012}
J.~Fraser.
\newblock On the packing dimension of box-like self-affine sets in the plane.
\newblock {\em {Nonlinearity}}, 25:2075--2092, 2012.

\bibitem[Goo41]{Goo41}
I.~Good.
\newblock The fractional dimensional theory of continued fractions.
\newblock {\em {Proc. Cambridge Philos. Soc.}}, 37:199--228, 1941.

\bibitem[Hen96]{Hensley_1996}
D.~Hensley.
\newblock A polynomial time algorithm for the {H}ausdorff dimension of
  continued fraction cantor sets.
\newblock {\em {J. Number Theory}}, 58:9--45, 1996.

\bibitem[Hut81]{Hut1981}
J.~Hutchinson.
\newblock Fractals and self-similarity.
\newblock {\em {Indiana Univ. Math. J.}}, 30:713--747, 1981.

\bibitem[Jar28]{Jar28}
I.~Jarnik.
\newblock Zur metrischen {T}heorie der diophantischen {A}pproximationen.
\newblock {\em {Prace Mat.-Fiz.}}, 36:91--106, 1928.

\bibitem[JR12]{JR12}
\textls[-15]{T.~Jordan and M.~Rams.}
\newblock \textls[-15]{Increasing digit subsystems of infinite iterated function systems.}
\newblock \textls[-15]{{\em {Proc.\,Amer.\,Math.\,Soc.}}, 140:1267--1279, 2012.}

\bibitem[Jur21]{Jurga_dimension_spectrum}
N.~Jurga.
\newblock Dimension spectrum of infinite self-affine iterated function systems.
\newblock {\em {Selecta Math.}}, 27, 2021.

\bibitem[KR14]{KR14}
A.~K\"{a}enm\"{a}ki and H.~Reeve.
\newblock Multifractal analysis of {B}irkhoff averages for typical infinitely
  generated self-affine sets.
\newblock {\em {J. Fractal Geom.}}, 1:83--152, 2014.

\bibitem[KM22]{KM22b}
C.~Kalle and M.~Maggioni.
\newblock On approximation by random {L}\"{u}roth expansions.
\newblock {\em {Int. J. Number Theory}}, 18:1013--1046, 2022.

\bibitem[KKK90]{KKK90}
S.~Kalpazidou, A.~Knopfmacher, and J.~Knopfmacher.
\newblock L\"{u}roth-type alternating series representations for real numbers.
\newblock {\em {Acta Arith.}}, 55:311--322, 1990.

\bibitem[KKK91]{KKK91}
\textls[-15]{S.~Kalpazidou,\,A.~Knopfmacher,\,and\,J.~Knopfmacher.}
\newblock \textls[-15]{Metric properties of alternating {L}\"{u}roth series.}
\newblock \textls[-15]{{\em {Portugal.\,Math.}}, 48:319--325, 1991.}

\bibitem[KSS95]{KSS95}
M.~Keane, M.~Smorodinsky, and B.~Solomyak.
\newblock On the morphology of {$\gamma$}-expansions with deleted digits.
\newblock {\em {Trans. Amer. Math. Soc.}}, 347:955--966, 1995.

\bibitem[KZ06]{Kess_Zhu_2006}
M.~Kesseböhmer and S.~Zhu.
\newblock Dimension sets for infinite {IFS}s: the {T}exan conjecture.
\newblock {\em {J. Number Theory}}, 116:230--246, 2006.

\bibitem[Lal97]{Lal97}
S.~Lalley.
\newblock {$\beta$}-expansions with deleted digits for {P}isot numbers
  {$\beta$}.
\newblock {\em {Trans. Amer. Math. Soc.}}, 349:4355--4365, 1997.

\bibitem[LWY18]{LWY18}
J.~Li, M.~Wu, and X.~Yang.
\newblock On the longest block in {L}\"{u}roth expansion.
\newblock {\em {J. Math. Anal. Appl.}}, 457:522--532, 2018.

\bibitem[L{\"u}r83]{Luroth1883}
J.~L{\"u}roth.
\newblock {\"U}ber eine eindeutige {E}ntwicklung von {Z}ahlen in eine
  unendliche {R}eihe.
\newblock {\em {Math. Ann.}}, 21:411--423, 1883.

\bibitem[Mar54]{Mar54}
J.~Marstrand.
\newblock The dimension of {C}artesian product sets.
\newblock {\em {Proc. Cambridge Philos. Soc.}}, 50:198--202, 1954.

\bibitem[MU96]{mauldin1996dimensions}
\textls[-11]{R.~Mauldin and M.~Urba{\'n}ski.}
\newblock \textls[-11]{Dimensions and measures in infinite iterated function systems.}
\newblock \textls[-11]{{\em {Proc.\,Lond.\,Math.\,Soc.}}, 3:105--154, 1996.}

\bibitem[MU99]{MU02}
R.~Mauldin and M.~Urba\'{n}ski.
\newblock Conformal iterated function systems with applications to the geometry
  of continued fractions.
\newblock {\em {Trans. Amer. Math. Soc.}}, 351:4995--5025, 1999.

\bibitem[MU03]{MR2003772}
R.~Mauldin and M.~Urba\`{n}ski.
\newblock {\em {Graph directed {M}arkov systems}}, volume 148 of {\em Cambridge
  Tracts in Mathematics}.
\newblock Cambridge University Press, Cambridge, 2003.

\bibitem[McM84]{MR0771063}
C.~McMullen.
\newblock The {H}ausdorff dimension of general {S}ierpi\'{n}ski carpets.
\newblock {\em {Nagoya Math. J.}}, 96:1--9, 1984.

\bibitem[Mor18]{Morris}
I.~Morris.
\newblock An explicit formula for the pressure of box-like affine iterated
  function systems.
\newblock {\em {J. Fractal Geom.}}, 6:127--141, 2018.

\bibitem[MS19]{MS19}
I.~Morris and P.~Shmerkin.
\newblock On equality of {H}ausdorff and affinity dimensions, via self-affine
  measures on positive subsystems.
\newblock {\em {Trans. Amer. Math. Soc.}}, 371:1547--1582, 2019.

\bibitem[PS95]{PS95}
\textls[-13]{M.~Pollicott and K.~Simon.}
\newblock \textls[-13]{The {H}ausdorff dimension of {$\lambda$}-expansions with deleted digits.}
\newblock \textls[-13]{{\em {Trans.\,Amer.\,Math.\,Soc.}}, 347:967--983, 1995.}

\bibitem[Rap23]{Rap_2023}
A.~Rapaport.
\newblock Dimension of diagonal self-affine sets and measures via non-conformal partitions.
\newblock arxiv:2309.03985v1, 2023.

\bibitem[RGU16]{rempe2016non}
L.~Rempe-Gillen and M.~Urba{\'n}ski.
\newblock Non-autonomous conformal iterated function systems and {M}oran-set
  constructions.
\newblock {\em {Trans. Amer. Math. Soc.}}, 368:1979--2017, 2016.

\bibitem[Zho22]{Zho22}
Q.~Zhou.
\newblock Dimension of exceptional sets arising by the longest block function
  in {L}\"{u}roth expansions.
\newblock {\em {J. Math. Anal. Appl.}}, 510:Paper No. 126011, 15, 2022.

\end{thebibliography}
\end{document}